\documentclass[sn-mathphys,Numbered]{sn-jnl}
\usepackage{graphicx}%
\usepackage{multirow}%
\usepackage{amsmath,amssymb,amsfonts}%
\usepackage{amsthm}%
\usepackage{mathrsfs}%
\usepackage[title]{appendix}%
\usepackage{xcolor}%
\usepackage{textcomp}%
\usepackage{manyfoot}%
\usepackage{booktabs}%
\usepackage{algorithm}%
\usepackage{algorithmicx}%
\usepackage{algpseudocode}%
\usepackage{listings}%
\usepackage{natbib}

\newtheorem{theorem}{Theorem}

\newtheorem{proposition}{Proposition}

\newtheorem{remark}{Remark}%

\newtheorem{lemma}{Lemma}
\begin{document}
	
	\title[Article Title]{Numerical approximation  of
		the boundary control for the wave equation with a spectral collocation method}

	\author*[1]{\fnm{Somia} \sur{Boumimez}}\email{soumia.boumimez@gmail.com}
	\author[1]{\fnm{Carlos} \sur{Castro}}\email{carlos.castro@upm.es}
	\affil*[1]{M2ASAI Universidad Polit\'ecnica de Madrid, Departamento de Matem\'atica e Inform\'atica, ETSI Caminos, Canales y Puertos, 28040 Madrid, Spain}
	
	\abstract{	We propose a spectral collocation method to approximate the exact boundary control of the wave equation in a square domain. The idea is to introduce a suitable approximate control problem that we solve in the finite-dimensional space of polynomials of degree $N\in \mathbb{N}$ in space. We prove that we can choose a sequence of controls $f^N$ associated to the approximate control problem in such a way that they converge, as $N\to \infty$, to a control of the continuous wave equation. Unlike other numerical approximations tried in the literature, this one does not require regularization techniques and can be easily adapted to other equations and systems where the controllability of the continuous model is known.  The method is illustrated with several examples in  1-d  and  2-d  in a square domain. We also give numerical evidences of the highly accurate approximation inherent to spectral methods.}
	\maketitle
	
	\section{Introduction}\label{sec1}

	Consider the wave equation on a square domain $\Omega=(-1,1)^d \subset \mathbb{R}^d$ ($d=1,2$) with a control $ f $ acting on one part of the boundary $\Gamma \subset \partial \Omega$ for some time $t\in (0,T)$:
	\begin{align}\label{eq1}
		\begin{cases}
			u_{tt}-\Delta u=0& \mbox{in} \ Q=(0,T)\times\Omega \\
			u=f & \mbox{on} \ (0,T)  \times \Gamma \\
			u=0 & \mbox{on} \ (0,T)  \times \partial \Omega \backslash \Gamma \\
			u(0,x)=u^{0}(x),\; u_{t}(0,x)=u^{1}(x)\hspace*{0.5cm}& \mbox{in} \ \Omega .
		\end{cases}
	\end{align}
	
	Given any $ f\in L^{2}((0,T)\times \Gamma ) $ and some initial data $(u^{0}, u^{1}) \in L^{2}(\Omega)\times H^{-1}(\Omega)  $,	 problem (\ref{eq1}) has a unique solution $ (u, u_{t}) \in C([0, T] , L^{2}(\Omega)\times H^{-1}(\Omega))$. 
	It is also well-known that, if $ T>T_0 $ with $T_0$ sufficiently large and $\Gamma\subset\partial \Omega$ satisfies some geometric conditions (see \citep{bib1}, \citep{bib15}), for  any initial data $ (u^0,u^1)\in L^2(\Omega)\times H^{-1}(\Omega)  $ there exists a control  $ f\in L^{2}((0,T)\times \Gamma ) $  such that the solution of (\ref{eq1}) $ u\in C([0,T];H_0^2(\Omega))\cap C^1([0,T];L^2(\Omega))) $ can be driven to any final target. We assume without loss of generality that this final target is the equilibrium, i.e.
	\begin{align}\label{eq3}
		\begin{split}
			u(T,x)=u_{t}(T,x)=0,\  x\in \Omega.
		\end{split}
	\end{align}	
	
	In particular this is true in dimension $d=1$ when $\Gamma$ is one extreme and $T_0=4$, and in dimension $d=2$ when $\Gamma$ is the union of two consecutive sides and $T_0=4\sqrt{2}$ (see  \cite{bib15}). It is important to note that the control, when it exists, is not unique in general. Among the set of controls, a natural choice is to consider the minimal $L^2-$norm control which is usually unique.

	In this work we focus on the numerical approximation of these boundary controls $f$. This problem has been extensively studied in the last decades with different numerical methods. In particular, it is well known that a discretization of system \eqref{eq1} with finite elements or finite differences schemes reduces the problem to a finite dimensional control problem which is not  uniformly controllable with respect to the discretization parameter. This means that the discrete control problem do not provide a bounded sequence of controls and therefore it is not possible to use this strategy to approximate the continuous control. This was first observed in \cite{bib11} where the authors considered a finite dimensional version of the Hilbert Uniqueness Method (HUM) introduced in \cite{bib15}. Since then, several cures have been proposed to recover convergence approximations of the controls as bigrid algorithms, Tychonoﬀ regularization, filtering, mixed finite elements, etc (see for instance ,\cite{bib6},\cite{bib7},\cite{bib10}, \cite{bib12} and the review paper \cite{bib20}). We also mention more recent approaches where controls are obtained by minimizing a cost function that penalizes both the control and the state \cite{bib8} or those based on a space-time formulation that does not require regularizations \cite{bib4}. 
	
	Here we propose a new numerical approach based on the spectral collocation method.	 For the background and details on this method as well as on other classes of spectral methods (Galerkin, collocation,...) we refer the reader to Gottlieb and Orszag \cite{bib13}, Canuto et al. \cite{bib5}, Bernardi and Maday \cite{bib2}.
	These methods have been extensively used in the past 30 years especially for the numerical
	simulation of fluid dynamical problems (e.g.\cite{bib5} ). According to this approach, the numerical
	solution is regarded as a smooth global polynomial of degree $N$ (typically quite large), and
	satisfies the equilibrium equations point-wise at a family of collocation points that are the
	nodes of a high precision Gauss-Lobatto integration formula. This process can be regarded as
	a generalized Galerkin method: the distinguishing feature with respect to the conventional
	finite element method is that trial (as well as test functions are global (rather than piecewise)
	polynomials of high degree. 
	Two important consequences are derived from this fact. The first is that the spectral method is
	potentially extremely accurate. Indeed, for problems with smooth data, the order of convergence of the numerical solution is much higher than that achievable by finite element
	approximations. To some extent this feature still holds even for problems with low
	smoothness solution, such as those arising in fracture mechanics, or whenever loads are
	concentrated on a small part of the boundary. The second consequence of the global character of the test functions is that the spectral matrixes are severely ill-conditioned and preconditioning techniques must be implemented for large scale problems. 
	
	The use of spectral methods to approximate the control of the wave equation has been previously investigated  by Boulmezaoud and Urquiza \cite{bib3} where, instead of collocation, a Galerkin spectral method is considered. The proposed approximation is not uniformly controllable. However, a bounded sequence of controls is obtained when trying to control the projection of the solution in a suitable low frequencies space, similar as the result obtained for finite differences in \cite{bib14}. From a practical point of view this is not satisfactory since it requires to know an accurate representation of the eigenfunctions associated to the discrete problem, something  which is not available in general. 
	
	The novelty here is that we are able to prove the uniform controllability, and therefore a convergent sequence of controls as $N\to \infty$, by adding an extra discrete boundary control that vanishes as $N\to \infty$. This provides an accurate approximation of the continuous control. The result relies on two key aspects: a uniform observability inequality for the associated discrete adjoint system and a detailed spectral analysis of the discrete low frequencies. The first property allows us to obtain the uniform boundedness of discrete controls while the second one is used to obtain the convergence of the discrete control to the continuous one. 
	
	The method we present here to obtain the uniform observability inequality is new and considers, instead of the discrete collocation system, the equivalent  continuous error equation associated to the polynomial approximation (see \cite{bib13}). This error equation is the same wave equation but with a nonhomogeneous second hand term, known as the error term. Therefore, the observability inequality can be derived using the same techniques as in the continuous model and we only have to estimate this extra error term. This is an important advantage of the method since it can be easily extended to more general equations (elasticity, fluid dynamics, etc.) and  higher dimensions, as long as we consider rectangular domains. 
	
	The second important advantage of the method is in the convergence rate of the approximation. Here we only prove that convergence holds but the numerical experiments illustrate that one recovers the high accuracy expected by a spectral method, even when nonsmooth data is considered. 
	
	To clarify the exposition we present detailed proofs in the one-dimensional case and the main results for the two dimensional one. As we mentioned before, the proofs can be easily adapted by separation of variables. 
	
	The rest of the paper is divided in four more sections. The second section is devoted to state and prove tha main results for the 1-d wave equation. The third section states the main results for the 2-d case. In section 4  we present some  numerical examples in $ 1$-d and in $ 2$-d  in a
	square domain. Finally, the Appendix contains the main spectral results required for the analysis of the convergence.

	\section{The 1-d wave equation}\label{sec2}
	In this section we focus on the 1-d wave equation. Here $ \Omega=(-1,1) $ and we assume that the control $ f  $ acts at the right extreme $ x = 1 $. System \eqref{eq1} reads, 
	\begin{align}\label{eq1-d}
		\begin{cases}
			u_{tt}-u_{xx}=0&\mbox{in} \ (t,x)\in \ (0,T)\times\Omega\\
			u(t,1)=f(t)& \mbox{in} \ t\in \ (0,T)\\
			u(t,-1)=0& \mbox{in} \ t\in \ (0,T)\\
			u(0,x)=u^{0},\quad u_{t}(0,x)=u^{1}\hspace*{0.5cm}&\mbox{in}\ x\in \ \Omega,
		\end{cases}
	\end{align}
	where $ (u^0,u^1)\in L^2(\Omega)\times H^{-1}(\Omega)  $ is the given  initial data.
	
	Assume that $T>4$. We are interested in approximating one of the controls $ f\in L^{2}(0,T) $ for which \eqref{eq3} is satisfied. More precisely, we follow HUM (see \cite{bib15}) and approximate the control that has minimal $L^2-$norm with respect to a suitable weighted norm, with a smooth weight $\eta(t)$ that is compactly supported bump function in $t\in(0,T)$. This ensures that the control itself is compactly supported and avoids possible singularities at times $t=0,T$. In order to find a numerical approximation of this control $ f $ in \eqref{eq1-d} we proceed as follow: first we introduce a discrete version of the control problem \eqref{eq1-d}, depending on a discrete parameter $ N\longrightarrow \infty $. Then, we prove that this system is controllable for all $ N $ with three different controls, $f^{N},g^{N}_{R},g^{N}_{L} \in L^{2}(0,T), $ that we can choose in such way that \[f^{N} \underset{N\to\infty}{\longrightarrow} f,\ g_{R}^{N}\underset{N\to\infty}{\longrightarrow} 0,\ g_{L}^{N}\underset{N\to\infty}{\longrightarrow} 0,\]
	where $ f $ is a control of \eqref{eq1-d}. Therefore $f^N$ is a numerical approximation of a continuous control $f$, while  $g^{N}_{R},g^{N}_{L}$ can be understood as artificial controls which are only necessary to obtain $f^N$.
	
	\subsection{Approximation by the spectral collocation method }
	In this section we introduce some notation and state the main results of the paper. 	
	
	Let $ N $ be a  natural number and consider $ C=\{x_{i},\ 0\le i\le N\} $ the Legendre-Gauss-Lobatto (LGL) nodes in $ \Omega $ that are the roots of 	
	\[(1-x^{2})\dfrac{d}{dx}L^N(x),\] 
	where $ L^k(t ) $ is the k-th Legendre polynomial in $ (-1, 1)  $ (e.g. \cite{bib17}).
	We divide $ C = C^{\Omega}\cup C^{Di}$ into interior and boundary nodes,
	i.e.
	\begin{align*}
		&C^{\Omega}=C\cap \Omega =\{x_{i},i\in I_{\Omega}\}\\
		&C^{Di}=C\cap \{-1,1\}=\{x_{i},i\in I_{Di}\},
	\end{align*}
	where $I_{\Omega},\; I_{Di}$ are the sets of indexes corresponding to the interior and boundary collocation nodes respectively, and we denote $ I=I_{\Omega}\cup I_{Di} $.  
	
	Let $ \mathbb{P}_{N}(\Omega) $ be the space of continuous functions in $ \bar{\Omega} $ which are polynomials of degree
	less than or equal to $ N $ and let $ \mathbb{P}_{N}^{Di}(\Omega)  $ be the subspace of $ \mathbb{P}_{N}(\Omega) $ of those functions vanishing on $ x=\{-1,1\}$. We define the following
	discrete inner product that approximates that of $ L^{2}(\Omega) $: 
	\begin{align}\label{inner1}
		(w,z)_{N}=\sum_{i\in I}(wz)(t,x_{i})\omega_{i},\ 0\le t\le T.
	\end{align}
	
	Here $ \omega_{i} $ is the discrete weight associated with the one-dimensional LGL quadrature formula (e.g.\cite{bib5}, Chapter 2]). Owing to the exactness of the integration LGL formula, we have
	\begin{align}\label{intn}
		(w,z)_{N}=\int_{\Omega} wz\ dx\ for\ all\ w,z\ such \ that \ wz\in\mathbb{P}_{2N-1}(\Omega).
	\end{align}
	
	The symbol $ \left\|\cdot \right\|_{N}$ denotes the discrete norm which is defined as $\left\|z \right\|_{N}^{2}=(z,z)_{N}$.
	
	We recall that the discrete norm $ \left\|\cdot\right\| _{N}  $ is uniformly equivalent to the $L^2-$norm $ \left|\cdot\right|_{L^{2}}  $ in $\mathbb{P}_{N}(\Omega)$ 
	(\cite{bib5}, Chapter 9). In other words, there exist two positive constants $ C_{1}=1 $, and $ C _{2}=2+\dfrac{1}{N}$, such that
	\begin{align}\label{eq10}
		C_{1}\left|p\right|^{2}_{L^{2}}\le\left\|p\right\|^{2}_{N}\le C _{2}\left|p\right|^{2}_{L^{2}}, \ \forall p\in \mathbb{P}_{N}(\Omega).	
	\end{align}
	We denote by $ \Psi_{i} $ the Lagrange polynomials which is $ 1 $ at $x_{i} $ and $ 0 $ at all the other collocation nodes. For the
	commonly used Gauss-Lobatto points one has
	\begin{align}
		\begin{split}
			\Psi_{i}(x)=\dfrac{1}{N(N+1)L_{N}(x_{i})}\dfrac{(1-x^{2})L_{N}(x)}{x-x_{i}}.
		\end{split}
	\end{align} 
	Observe that $\{\Psi_i,\; i\in I\}$ constitutes a basis in $ \mathbb{P}_{N}^{Di}(\Omega)  $.	
	
	Now we introduce the following discrete control problem: Given $u^{0,N},\;u^{1,N} \in \mathbb{P}^{Di}_{N}(\Omega)$ and  $T>0$, find $ f^{N},\;g_{R}^{N},\;g_{L}^{N}\in L^{2}(0,T) $ such that the solution $ u^{N}\in C^{\infty}(0,T;\mathbb{P}_{N}(\Omega)) $  of system:
	\begin{align}\label{d11-d}
		\begin{split}
			\begin{cases}
				(u^{N}_{tt}-u^{N}_{xx})(t,x_{i})=  g_{L}^{N}(t)G^{N}_{L}(x_{i})+g_{R}^{N}(t)G^{N}_{R}(x_{i})&\mbox{in}\ (t,x_{i})\in (0,T)\times C^{\Omega}\\
				u^{N}(t,1)=f^{N}(t)&\mbox{in}\ t\in \ (0,T)\\
				u^{N}(t,-1)=0& \mbox{in}\ t\in \ (0,T)\\
				u^{N}(0,x_{i})=u^{0,N}(x_i),u_{t}^{N}(0,x_{i})=u^{1,N}(x_i)&\mbox{ for }\ x_{i}\in C^{\Omega},
			\end{cases}
		\end{split}
	\end{align}
	satisfies
	\begin{align}\label{eq141-d}
		\begin{split}
			u^{N}(T,x_{i})=u^{N}_{t}(T,x_{i})=0,\ x_{i}\in C^{\Omega}.
		\end{split}
	\end{align} 
	Here $ G^{N}_{L},G^{N}_{R} \in \mathbb{P}_{N-1}(\Omega)$ are defined by 
	\begin{align}\label{condc1}
		\begin{split}
			\begin{cases}
				G^{N}_{L}(x_{i})=\bigg(\dfrac{h^{L}_{xx}}{\sqrt{\omega_{0}}}-\dfrac{\Psi_{0,x}}{\sqrt{\omega_{0}}\omega_{0}}\bigg)(x_{i}),\quad G^{N}_{R}(x_{i})
				= \bigg(\dfrac{h^{R}_{xx}}{\sqrt{\omega_{N}}}+\dfrac{\Psi_{N,x}}{\sqrt{\omega_{N}}\omega_{N}}\bigg)(x_{i})\\
				h^{L},\ h^{R}\in \mathbb{P}_{N}^{Di}(\Omega)\\
				h^{L}(x_{i})=\dfrac{1-x_{i}}{2},\ h^{R}(x_{i})=\dfrac{1+x_{i}}{2},\ x_{i}\in C^{\Omega}.
			\end{cases}
		\end{split}
	\end{align} 
	Note that $h^{L}(x)\neq\dfrac{1-x}{2}$ since $h^{L}\in \mathbb{P}_{N}^{Di}(\Omega)$. Something similar can be said about $h^{R}(x)$. 	
	
	The main results in this paper are the following :
	\begin{theorem}\label{th1}
		Given $T>4(2+N^{-1})$ and $(u^{0,N}, \; u^{1,N}) \in \mathbb{P}^{Di}_{N}(\Omega)\times \mathbb{P}^{Di}_{N}(\Omega)$, there exist controls $f^N,\; g^{N}_{L},\; g^{N}_{R} \in L^2(0,T)$ such that the solution $u^N$ of 
		\eqref{d11-d} satisfies \eqref{eq141-d}.
	\end{theorem}	
	
	\begin{theorem}\label{thcon1-d}
		Given $(u^{0}, u^{1}) \in L^2\times H^{-1}$, there exists a sequence $(u^{0,N}, u^{1,N}) \in (\mathbb{P}^{Di}_{N}(\Omega))^2$ such that  
		\[
		(u^{0,N}, u^{1,N}) \to (u^{0}, u^{1}) \;\mbox{in}\; L^2\times H^{-1}, \quad \mbox{ as $N\to \infty$}.\]
		Furthermore, for  any $T>4(2+N^{-1})$, we can choose the controls $f^N,\; g^{N}_{L},\; g^{N}_{R} \in L^2(0,T)$ such that the solution $u^N$ of \eqref{d11-d} satisfies \eqref{eq141-d} and
		$$
		f^N \to f, \quad g^{N}_{R} \to 0,\quad g^{N}_{L} \to 0,\quad \mbox{ as $N\to \infty$,  in } L^2(0,T),
		$$ 
		where $f$ is a control of the continuous wave equation \eqref{eq1-d}. 
		
		When $u^0$ (resp. $u^1$) is a continuous functions we can just take $u^{0,N}\in \mathbb{P}_N^{Di}$ such that $u^{0,N}(x_i)=u^0(x_i)$ (resp. $u^{1,N}(x_i)=u^1(x_i)$).
	\end{theorem}		
	
	\begin{remark}\label{remark1}
		Note that the control time $T$ in Theorem \ref{thcon1-d} is basically two times the time required in the continuous problem. This is due to the constant $C_2$ in  \eqref{eq10} and probably not optimal, as we illustrate in the experiments below. 
	\end{remark}
	
	\subsection{Existence of discrete controls: proof of Theorem \ref{th1}}
	
	In this section we prove Theorem \ref{th1}. We first introduce a variational characterization of discrete controls \eqref{d11-d} and then prove that a particular discrete control can be obtained as the minimizer of a convex quadratic functional defined on a polynomial space. Finally, we prove the coerciveness of the functional that guarantees the existence of minimizers. 
	
	Let us introduce the following bilinear form in $ \mathbb{P}^{Di}_{N}(\Omega)\times \mathbb{P}^{Di}_{N}(\Omega) $,
	\begin{align}\label{dualityN}
		\left\langle (\phi^{0,N},\phi^{1,N}),(u^{0,N},u^{1,N})\right\rangle _{N}=(u^{1,N},\phi^{0,N})_{N}-(u^{0,N},\phi^{1,N})_{N} .
	\end{align}
	
	\begin{lemma}\label{lemvar1-d} Assume that $ T>0$, and consider some initial data $(u^{0,N},u^{1,N})\in \mathbb{P}^{Di}_{N}(\Omega)\times \mathbb{P}^{Di}_{N}(\Omega)$.
		Any controls $  f^{N},g_{R}^{N},g_{L}^{N}$ that make the solution of the discrete system \eqref{d11-d} satisfy \eqref{eq141-d} are solutions of,
		\begin{align}\label{key11-d}
			\begin{split}
				&	\int_{0}^{T}(\phi^{N}_{x}(t,1)-\omega_{N}\phi^{N}_{xx}(t,1))f^{N}(t) dt+\int_{0}^{T}\sqrt{\omega_{N}}\phi^{N}_{xx}(t,1)g_{R}^{N}(t) dt\\&\hspace*{0.5cm}
				+\int_{0}^{T}\sqrt{\omega_{0}}\phi^{N}_{xx}(t,-1)g_{L}^{N}(t) dt-\big<(\phi^{N}(0,\cdot),\phi^{N}_{t}(0,\cdot)),(u^{0,N},u^{1,N})\big>_{N}=0,
			\end{split}
		\end{align}
		for all $(\phi^{0,N},\phi^{1,N})\in \mathbb{P}^{Di}_{N}(\Omega)\times \mathbb{P}^{Di}_{N}(\Omega)$, where $ (\phi^{N},\phi^{N}_{t}) \in \mathbb{P}^{Di}_{N}(\Omega)\times \mathbb{P}^{Di}_{N}(\Omega) $ is the solution of the following collocation backwards wave equation :
		\begin{align}\label{eq1111-d}
			\begin{cases}
				(\phi^{N}_{tt}-\phi^{N}_{xx})(t,x_{i})=0&\mbox{in}\ (t,x_{i})\in (0,T)\times C^{\Omega}\\
				\phi^{N}(t,1)=\phi^{N}(t,-1)=0 &\mbox{in} \ t\in (0,T)\\
				\phi^{N}(T,x_{i})=\phi^{0,N}(x_i),\quad \phi_{t}^{N}(T,x_{i})=\phi^{1,N}(x_i)&\mbox{at}\ x_{i}\in C^{\Omega}.
			\end{cases}
		\end{align}
	\end{lemma}
	\begin{proof}[\textbf{Proof.}]
		Multiplying the equation of $ u^{N}(t,x_i) $ in \eqref{d11-d} by $ \omega_{i} \phi^{N}(t,x_i)$ and adding in $i\in I$ one obtains,
		\begin{align}\label{var2}
			\begin{split}
				\int_{0}^{T}(u^{N}_{tt}- u_{xx}^{N},\phi^{N})_{N}dt
				= \int_{0}^{T}g_{L}^{N}\; ( G^{N}_{L},\phi^{N})_{N}dt+\int_{0}^{T}g_{R}^{N} \; (G^{N}_{R},\phi^{N})_{N}dt.
			\end{split}
		\end{align}
		We first simplify the left hand side. Using \eqref{intn}, integrating by parts in time and taking into account that $f^N,$ $g^N_L$ and $g^N_R$ are controls we have 
		\begin{align}\label{left12}
			\begin{split}
				&\int_{0}^{T}(u^{N}_{tt}-u_{xx}^{N},\phi^{N})_{N}dt=\int_{0}^{T}(u^{N},\phi^{N}_{tt})_{N}dt-\int_{0}^{T}\int_{-1}^{1}u^{N}_{xx}\phi^{N}dxdt\\
				&\qquad-\left\langle (\phi^{N}(0,.),\phi^{N}_{t}(0,.)),(u^{0,N},u^{1,N})\right\rangle  _{N}.
			\end{split}
		\end{align}
		We now	integrate by parts in $x$ and use again formula \eqref{intn}, since the resulting integrand is also a polynomial of degree $2N-2$,	
		\begin{align}\label{left1}
			\begin{split}
				&0=\quad\int_{0}^{T}(u^{N}_{tt}-u_{xx}^{N},\phi^{N})_{N}dt =\int_{0}^{T}(u^{N},\phi^{N}_{tt}-\phi_{xx}^{N})_{N}dt\\&\quad + \int_0^T f^N(t)\phi_x^N(t,1)dt-\left\langle (\phi^{N}(0,.),\phi^{N}_{t}(0,.)),(u^{0,N},u^{1,N})\right\rangle  _{N}\\
				&=\int_{0}^{T}f^{N}(t)(\phi^{N}_{x}-\omega_{N}\phi^{N}_{xx})(t,1)dt-\left\langle (\phi^{N}(0,.),\phi^{N}_{t}(0,.)),(u^{0,N},u^{1,N})\right\rangle  _{N}.
			\end{split}
		\end{align}
		The last equality is a consequence of the first equation in \eqref{eq1111-d}. Note that an extra term appears in the right hand side of this expression coming from the fact that the first equation in \eqref{eq1111-d} is only true for the interior nodes while the discrete scalar product involves also the boundary nodes.
		
		For the right hand side in \eqref{var2} we use again formula \eqref{intn} and the fact that $G_R^N \phi^N$ is a polynomial of degree $2N-1$,
		\begin{align} \label{eq_c1}
			\begin{split}
				(G^{N}_{R},\phi^{N})_{N}=	&\dfrac{1}{\sqrt{\omega_{N}}} \left( h^{R}_{xx}+\dfrac{\Psi_{N,x}}{\omega_{N}},\phi^{N} \right)_{N}
				=\dfrac{1}{\sqrt{\omega_{N}}}  \int_{-1}^{1}\left( h^{R}_{xx}+\dfrac{\Psi_{N,x}}{\omega_{N}} \right)\phi^{N} dx\\
				&=\dfrac{1}{\sqrt{\omega_{N}}} \left( \int_{-1}^{1}h^{R}\phi^{N}_{xx}dx-\int_{-1}^{1}\dfrac{\Psi_{N}}{\omega_{N}}\phi^{N}_{x}dx\right) \\
				&=\dfrac{1}{\sqrt{\omega_{N}}} \left( \sum_{i\in  I_{\Omega}}(h^{R}\phi^{N}_{xx})(t,x_{i})\omega_{i}-\phi^{N}_{x}(t,1)\right) dt.
			\end{split}
		\end{align}
		To simplify the first term in the right hand side we observe that $h^R(x_i)=(1+x_i)/2$ at the interior nodes, i.e. $i\in I_\Omega$. Then, 
		\begin{align} \label{eq_c2}
			\begin{split}				
				&\sum_{i\in  I_{\Omega}}(h^{R}\phi^{N}_{xx})(t,x_{i})\omega_{i}= \int_{-1}^{1}\frac{1+x}{2}\phi^{N}_{xx}\; dx-\omega_{N}\phi^{N}_{xx}(t,1)=
				\phi^{N}_{x}(t,1)-\omega_{N}\phi^{N}_{xx}(t,1).
			\end{split}
		\end{align}		
		From \eqref{eq_c1}-\eqref{eq_c2} we easily obtain,
		\begin{equation} \label{eq_c3}
			(G^{N}_{R},\phi^{N})_{N}=-\sqrt{\omega_N}\phi^{N}_{xx}(t,1). 
		\end{equation}
		Combining \eqref{var2}, \eqref{left1} and \eqref{eq_c3} we easily find \eqref{key11-d}.
	\end{proof} 
	
	According to HUM, one possibility to construct controls  $ f^{N},g_{R}^{N},g_{L}^{N} $ that satisfy the variational condition \eqref{eq141-d} is as minimizers of the following cost functional   $ J^{N}: \mathbb{P}^{Di}_{N}(\Omega)\times \mathbb{P}^{Di}_{N}(\Omega)\longrightarrow \mathbb{R} $ defined  by
	\begin{align}\label{eq201-d}
		\begin{split}
			& J^{N}(\phi^{0,N},\phi^{1,N})=\dfrac{1}{2}\int_{0}^{T}
			\eta(t)\left| \phi^{N}_{x}(t,1)-\omega_{N}\phi^{N}_{xx}(t,1)\right| ^{2}dt\\
			&\quad +\dfrac{1}{2}\int_{0}^{T}\eta(t)\omega_{N}\left| \phi_{xx}^{N}(t,1)\right| ^{2}dt+\dfrac{1}{2}\int_{0}^{T}\eta(t)\omega_{0}\left| \phi_{xx}^{N}(t,-1)\right| ^{2}dt\\
			&\quad -\big<(\phi^{N}(0,\cdot),\phi^{N}_{t}(0,\cdot)),(u^{0,N},u^{1,N})\big>_{N},
		\end{split}
	\end{align}
	where $ (\phi^{N},\phi^{N}_{t}) $ is the solution  of \eqref{eq1111-d} with final data  $ (\phi^{0,N},\phi^{1,N})\in \mathbb{P}^{Di}_{N}(\Omega)\times \mathbb{P}^{Di}_{N}(\Omega)$. The function $ \eta(t) $ is a
	prescribed smooth function in $ [0, T] $ introduced to guarantee that the controls vanish in a neighborhood
	of $ t = 0, T $. Thus, we consider $ \delta > 0 $  a small number and $0\leq \eta(t)\leq 1$ such that,\begin{align}\label{hypeta}
		\eta(t)=\begin{cases}
			1 \;\mbox{in}\;[2\delta,T-2\delta]\\
			0\;\mbox{in}\;[0,\delta]\cup[T-\delta,T].
		\end{cases}
	\end{align}
	Note that both $\eta$ and $J^N$ will depend on this parameter $\delta$ but this is not relevant in the rest of the analysis and we will not make explicit this dependence in the notation. 
	
	\begin{theorem}\label{thdefc1-d} Assume $ (u^{0,N},u^{1,N})\in \mathbb{P}^{Di}_{N}(\Omega)\times \mathbb{P}^{Di}_{N}(\Omega)$ and that $(\hat{\phi}^{0,N},\hat{\phi}^{1,N})\in \mathbb{P}^{Di}_{N}(\Omega)\times \mathbb{P}^{Di}_{N}(\Omega)$ is a minimizer of $ J^{N}.$ If $ \hat{\phi^{N}} $ is the corresponding solution of \eqref{eq1111-d} with final data $ (\hat{\phi}^{0,N},\hat{\phi}^{1,N}) $ then
		\begin{align}\label{eq181-d}
			\begin{split}
				&f^{N}(t)=\eta(t)( \hat{\phi^{N}_{x}}(t,1)-\omega_{N}\hat{\phi^{N}_{xx}}(t,1))\\
				&g^{N}_{R}(t)=\eta(t)\sqrt{\omega_{N}}\hat{\phi^{N}_{xx}}(t,1), \qquad g^{N}_{L}(t)=\eta(t)\sqrt{\omega_{0}}\hat{\phi^{N}_{xx}}(t,-1),
			\end{split} 
		\end{align}
		are controls such that the solution of \eqref{d11-d} satisfies \eqref{eq141-d}.
	\end{theorem}
	\begin{proof}[\textbf{Proof.}]		
		If $ J^{N} $ achieves its minimum at $ (\hat{\phi}^{0,N},\hat{\phi}^{1,N}) $ its Gateaux derivative in the direction  $ (\phi^{0,N},\phi^{1,N}) $ must vanish, i.e.
		\begin{align*}
			0&= \int_{0}^{T}\eta(t)(\hat{\phi^{N}_{x}}(t,1)-\omega_{N}\hat{\phi^{N}_{xx}}(t,1))({\phi^{N}_{x}}(t,1)-\omega_{N}{\phi^{N}_{xx}}(t,1)) dt\\
			&\ +\int_{0}^{T}\eta(t)\omega_{N}\hat{\phi^{N}_{xx}}(t,1) \phi_{xx}^{N}(t,1)dt +\int_{0}^{T}\eta(t)\omega_{0}\hat{\phi^{N}_{xx}}(t,-1) \phi_{xx}^{N}(t,-1)dt\\
			&\ -\big<(\phi^{N}(0,\cdot),\phi^{N}_{t}(0,\cdot)),(u^{0,N},u^{1,N})\big>_{N},
		\end{align*}
		for all $ (\phi^{0,N},\phi^{1,N})\in \mathbb{P}^{Di}_{N}(\Omega)\times \mathbb{P}^{Di}_{N}(\Omega) $. From  Lemma \ref{lemvar1-d} it follows that \eqref{eq181-d} are controls for which \eqref{d11-d} holds.
	\end{proof}
	
	The functional $J^N$ is clearly continuous and convex so that the existence of a minimizer (and therefore a discrete control for system  \eqref{d11-d}) is guaranteed as soon as we prove its coercivity. This is  a consequence of the following lemma. Note that this also concludes the proof of Theorem \ref{th1}.


	\begin{lemma}\label{definv1-d}
		Given $ T > 4(2+N^{-1})$ there exists a constant $  C > 0 $, independent of $N$, such that the solutions of system \eqref{eq1111-d} satisfy  
		\begin{align}\label{eq2441-d}
			\begin{split}
				&C\left\|(\phi^{0,N}_x,\phi^{1,N}) \right\|^{2}_{N\times N}\le \int_{0}^{T}\left| \phi^{N}_{x}(t,1)-\omega_{N}\phi^{N}_{xx}(t,1)\right|^{2} dt\\
				&\qquad +\omega_{N}\int_{0}^{T}\left| \phi^{N}_{xx}(t,1)\right|^{2} dt+\omega_{0}\int_{0}^{T}\left| \phi^{N}_{xx}(t,-1)\right|^{2} dt,
			\end{split}
		\end{align}
		for any final data $ (\phi^{0,N},\phi^{1,N})\in \mathbb{P}^{Di}_{N}(\Omega)\times \mathbb{P}^{Di}_{N}(\Omega)$ and $ \eta(t) $ is defined as in \eqref{hypeta}.
	\end{lemma}

	\begin{proof}[\textbf{Proof.}]
		We prove the following equivalent version of \eqref{eq2441-d}
		\begin{eqnarray} \nonumber
			C\left\|(\phi^{0,N}_x,\phi^{1,N}) \right\|^{2}_{N\times N}&\le& \int_{0}^{T}\left| \phi^{N}_{x}(t,1)\right|^{2} dt+\omega_{N}\int_{0}^{T}\left| \phi^{N}_{xx}(t,1)\right|^{2} dt\\ \label{eq2441-d_bis}
			&& +\omega_{0}\int_{0}^{T}\left| \phi^{N}_{xx}(t,-1)\right|^{2} dt.
		\end{eqnarray}
		
		We first observe that the solutions of  \eqref{eq1111-d} solve the following equivalent system: 	
		\begin{align}\label{eq151-d}
			\begin{cases}
				\phi^{N}_{tt}-\phi^{N}_{xx}=-\phi_{xx}^{N}(t,-1)\Psi_{0}-\phi_{xx}^{N}(t,1)\Psi_{N}&\mbox{in} \ (t,x)\in (0,T)\times\Omega \\
				\phi^{N}(t,1)=\phi^{N}(t,-1)=0&\mbox{in} \ t\in (0,T)\\
				\phi^{N}(T,x)=\phi^{0,N},\phi_{t}^{N}(T,x)=\phi^{1,N}&\mbox{in} \ x\in\Omega.
			\end{cases}
		\end{align} 
		In fact, this is easily seen by writing 
		the polynomial $\phi^N_{tt}-\phi^N_{xx}$ in the Lagrangian basis and using system \eqref{eq1111-d}.	
		Now, we try the classical multiplier technique to recover the observability inequality. Multiplying  the equation in \eqref{eq151-d} by $ (x+1)\phi_{x}^{N} $ and integrating by parts one obtain
		\begin{align}\label{inv1-d}
			\begin{split}
				&\int_{0}^{T}\int_{-1}^{1}(x+1)(\phi^{N}_{tt}-\phi^{N}_{xx})\phi_{x}^{N}dx dt=-\int_{0}^{T}\int_{-1}^{1}(x+1)(\phi_{xx}^{N}(t,-1)\Psi_{0}\\
				&\qquad +\phi_{xx}^{N}(t,1)\Psi_{N})\phi_{x}^{N}dx dt.
			\end{split}
		\end{align}
		As for the continuous wave equation, (see \cite{bib15}) the left hand side can be simplified as follows,
		\begin{align}\label{inv1-d2}
			\begin{split}
				&\hspace*{0.5cm}\int_{0}^{T}\int_{-1}^{1}(x+1)(\phi^{N}_{tt}-\phi^{N}_{xx})\phi_{x}^{N}dx dt\\&=\int_{-1}^{1}(x+1)\phi^{N}_{t}\phi_{x}^{N}dx\bigg| _{0}^{T}+\dfrac{1}{2}\int_{0}^{T}\int_{-1}^{1}(\left| \phi_{t}^{N}\right| ^{2}+\left| \phi_{x}^{N}\right| ^{2})dx dt-\int_{0}^{T}\left| \phi_{x}^{N}(t,1)\right| ^{2}dt.
			\end{split}
		\end{align}
		Thus, if we set $ X=\int_{-1}^{1}(x+1)\phi^{N}_{t}\phi_{x}^{N}dx\bigg|_{0}^{T}$ 
		then 
		\begin{align}\label{lions1}
			\begin{split}
				&\dfrac{1}{2}\int_{0}^{T}\int_{-1}^{1}(\left| \phi_{t}^{N}\right| ^{2}+\left| \phi_{x}^{N}\right| ^{2})dx dt\le \left| X\right| + \int_{0}^{T}\left| \phi_{x}^{N}(t,1)\right| ^{2}dt \\&
				+\left| \int_{0}^{T}\phi_{xx}^{N}(t,-1)\int_{-1}^{1}(x+1)
				\phi^{N}_{x}\Psi_{0} dx dt\right| +\left| \int_{0}^{T}\phi_{xx}^{N}(t,1)\int_{-1}^{1}(x+1)\phi^{N}_{x}\Psi_{N} dxdt\right|.
			\end{split} 
		\end{align}
		In the rest of this proof we estimate each one of the terms in this expression. 
		
		We start with the left hand side in \eqref{lions1}. Define the discrete energy 
		\begin{equation} \label{eq_energy}
			E^{N}=\dfrac{1}{2}(\left\| \phi^{N}_{t}\right\| _{N}^{2}+\left\| \phi_{x}^{N}\right\|^{2} _{N}) . 
		\end{equation}
		It is easy to see that this energy is conserved, i.e. $ E^{N}_{t}= E^{N}_{0}$ for all $t>0$. We just multiply \eqref{eq1111-d} by $ \phi^{N}_{t} \omega_{i} $, add in $i\in I$ and integrate with respect time. This, together with the norm equivalence in \eqref{eq10} gives, 
		\begin{align*}
			& \dfrac{1}{2}\int_{0}^{T}\int_{-1}^{1}(\left| \phi_{t}^{N}\right| ^{2}+\left| \phi_{x}^{N}\right| ^{2})dx dt \geq \frac1{C_2}\int_{0}^{T}E^N_tdx dt = \frac{T}{C_2} E^N_T.
		\end{align*} 
		
		We now estimate the terms in the right hand side of \eqref{lions1}. We start with $X$. Observe that,
		\begin{equation}
			\left|\int_{-1}^{1}(x+1)\phi^{N}_{t}\phi_{x}^{N} dx \right|\le 2 \dfrac{1}{2}( \left\|\phi^{N}_{t} \right\|^{2}_{L^2}+\left\|\phi^{N}_{x} \right\|^{2}_{L^2})
			\le ( \left\|\phi^{N}_{t} \right\|^{2}_{ N}+\left\|\phi^{N}_{x} \right\|^{2}_{N})= 2 E^{N}_t = 2E^N_T.
		\end{equation} 
		Therefore  $
		\left|X\right|\le 4 E^{N}_{T}$.	
		Let us turn to estimate the third term in the right hand side in \eqref{lions1}. Using Young's inequality, for any $\varepsilon>0$ we can find $C_\varepsilon>0$ such that  
		\begin{align}
			\begin{split}
				&\hspace*{0.5cm}\left| \int_{0}^{T}\phi_{xx}^{N}(t,-1)\int_{-1}^{1}(x+1)
				\phi^{N}_{x}\Psi_{0} dx dt\right| \\ 
				&\le C_\varepsilon \left| \Psi_{0}\right|_{L^{2}(\Omega)}^{2}\left|\phi_{xx}^{N}(t,-1) \right|^{2}_{L^{2}(0,T)}+\varepsilon \left| \left|
				\phi^{N}_{x}\right| _{L^{2}(\Omega)}\right|_{L^{2}(0,T)}^{2}. 
			\end{split} 
		\end{align}
		Taking into account the norm equivalence in \eqref{eq10}, the conservation of the discrete energy proved above and the fact that, as 
		$\Psi_{0}\in \mathbb{P}_{N}(\Omega)$, $ \left| \Psi_{0}\right|_{L^{2}(\Omega)}^{2}\le\left\| \Psi_{0}\right\|_{N}^{2}= \omega_{0} $, we obtain
		\begin{align}\label{lions4}
			\begin{split}
				&\left| \int_{0}^{T}\phi_{xx}^{N}(t,-1)\int_{-1}^{1}(x+1)
				\phi^{N}_{x}\Psi_{0} dx dt\right|  \le C_\varepsilon \omega_{0}\left|\phi_{xx}^{N}(t,-1) \right|^{2}_{L^{2}(0,T)}+2\varepsilon TE_{T}^{N}.
			\end{split}
		\end{align}
		An analogous estimate holds for the last term in the right hand side in \eqref{lions1}.
		
		It follows from \eqref{lions1}, \eqref{lions4} and the fact that $\omega_{N}=\omega_{0}$,
		\begin{align*}
			\left(\dfrac{T}{C_{2}}-4-4\varepsilon T\right) E^{N}_{0}&\le \int_{0}^{T}\left| \phi_{x}^{N}(t,1)\right| ^{2}dt+\omega_{N} C_\varepsilon \left|\phi_{xx}^{N}(t,1) \right|^{2}_{L^{2}(0,T)}\\
			&+\omega_{N} C_\varepsilon \left|\phi_{xx}^{N}(t,-1) \right|^{2}_{L^{2}(0,T)}.
		\end{align*}
		Then, inequality \eqref{eq2441-d_bis} holds as long as $\frac{T}{C_2}-4\varepsilon T-4>0$. As $\varepsilon $ can be chosen arbitrarily small and $C_2=2+1/N$ we have the condition $T>4(2+N^{-1})$.
	\end{proof} 
	
	\begin{remark}
		The method we present here to obtain the uniform observability inequality relies on the continuous error equation \eqref{eq151-d}, equivalent to \eqref{eq1111-d}. The observability inequality is derived using the same techniques as in the continuous model (see  \cite{bib15}) and we only have to estimate the extra error term appearing in \eqref{eq151-d}. Therefore, this can be easily adapted to other equations or systems and higher dimensions where the controllability of the continuous model is known.  
	\end{remark}	
	
	\begin{remark}  Estimate \eqref{eq2441-d_bis} is known as an observability inequality. Roughly speaking it establishes that the energy of the solutions can be bounded by boundary observations, i.e. quantities that depend only on the solution at the boundary of the domain. Note that the uniformity in $N$ of the constant in \eqref{eq2441-d_bis} is not necessary to prove Theorem \ref{th1}, but it is essential to establish the convergence result in Theorem \ref{thcon1-d}, as we show in the next section.

		Note also that, at least formally, the terms involving $\varphi^N_{xx}(t,\pm1)$ in the right hand side of the observability inequality \eqref{eq2441-d}  should disappear as $N\to \infty$, since for the backwards wave equation \eqref{eq41-d} $\varphi_{xx}(t,\pm 1)=0$. However, if we remove these terms in \eqref{eq2441-d} the constant $C$ cannot be chosen to be uniform in $N\to \infty$. In fact,  for any $ T>0 $, we have 
		\begin{align}\label{nouniform}
			\sup_{(\phi^{0,N},\phi^{1,N})\in \mathbb{P}^{Di}_N\times\mathbb{P}^{Di}_N}\dfrac{\left\| (\phi^{0,N}_x,\phi^{1,N})\right\|^{2}_{N\times N}}{\int_{0}^{T}\left| \phi^{N}_{x}(t,1)\right|^{2} dt}\longrightarrow \infty\ \mbox{as}\ N\longrightarrow\infty.
		\end{align}
		Two spectral properties are at the origin of this lack of uniformity. One one hand the high frequency modes associated to the corresponding spectral collocation second-order differential operator
		\begin{align}\label{eigp11}
			\begin{split}
				\begin{cases}
					\varphi^{N}_{xx}(x_{i})+\lambda^{N}\varphi^{N}(x_{i})=0,&\ \mbox{ at } \ x_{i}\in C^{\Omega} \\
					\varphi^{N}(-1)=\varphi^{N}(1)=0.
				\end{cases}
			\end{split}
		\end{align}
		The eigenvalues associated to the continuous problem are given by $\lambda_k=(k \pi/2)^2$ and they satisfy a spectral gap property
		$$
		\sqrt{\lambda_{k+1}}-\sqrt{\lambda_k}=\pi/2 >0. 
		$$ 
		However, a numerical dispersion appears in the discrete eigenvalue problem and the analogous spectral gap is only true for the low frequencies. In figure \ref{fig2} we show the behavior of the square root of eigenvalues $\sqrt{\lambda_k^N}$ associated to \eqref{eigp11}. 
		\begin{figure}[h]
			\centering
			\includegraphics[width=12cm]{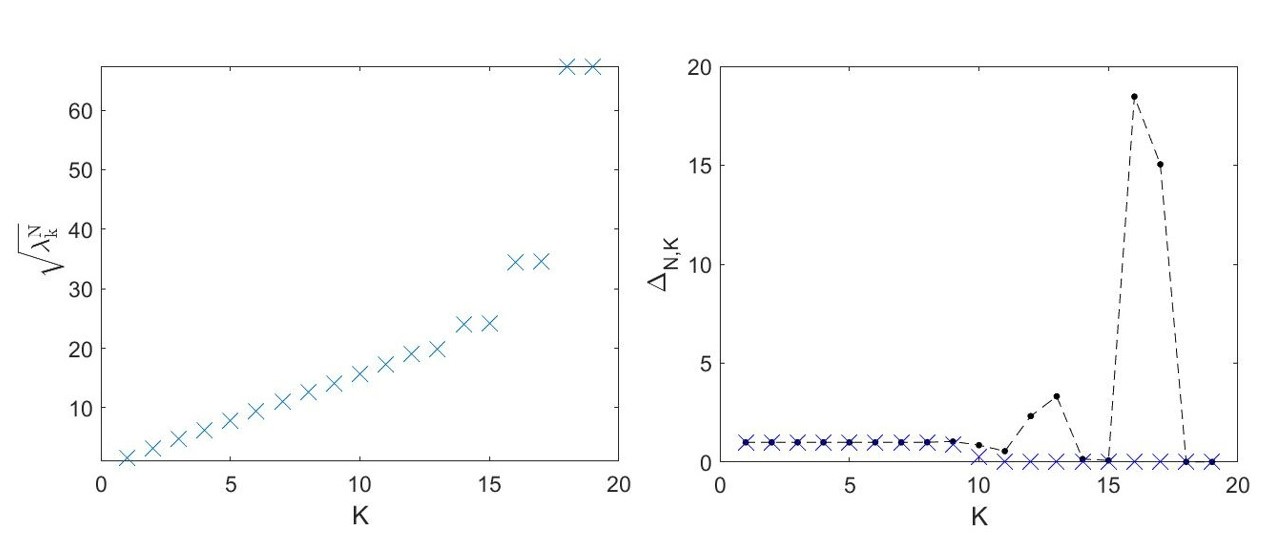}\\
			\caption{The left figure clarify the eigenvalues of \eqref{eigp11} for $ N=20 $ and the right one illustrate the behavior of the ratio \eqref{obs_rat} versus $ k $ for $ N=20 $ (dots in black), and the reinforced observability \eqref{obs_rat2} (crosses in blue)}
			\label{fig2}
		\end{figure}
		
		We observe that there are two families of eigenvalues, corresponding to even and odd eigenfunctions, that become closer for large frequencies. In particular, we find numerically 
		$$
		\sqrt{\lambda^{N}_{2k+1}}-\sqrt{\lambda^{N}_{2k}}\sim N^{-1.6}. 
		$$ 
		This lack of a uniform spectral gap for the discrete problem is one of the reasons to have \eqref{nouniform} (we refer to \cite{bib20} where this is analyzed in detail for other numerical approximations). This phenomena was already observed in \cite{bib19}.
		
		On the other hand, there is another spectral property of the continuous problem that is lost in the discretization and justifies \eqref{nouniform}. When considering particular solutions containing one single discrete eigenfunction $\varphi_k^{N}$ the left hand side in \eqref{nouniform} can be interpreted as an observability quotient for the eigenfunctions. This is uniformly bounded for all the eigenfunctions associated to the continuous eigenvalue problem but in the discrete case this property is lost. In particular, we have numerically 
		\begin{equation} \label{obs_rat}
			\max_{1\leq k\leq N}\dfrac{\sum_{i\in I} \omega_i \left| \varphi^N_{k,x}(x_i)\right| ^2}{\left| \varphi^N_{k,x}(1)\right| ^2}\sim N^{1.7}. 
		\end{equation}
		It turns out that, when considering the reinforced observation that includes the second derivative in space at the extremes we recover this uniformity, 		
		\begin{equation} \label{obs_rat2}
			\sup_{1\leq k\leq N}\dfrac{\sum_{i\in I} \omega_i\left| \varphi^N_{k,x}(x_i)\right| ^2}{\left| \varphi^N_{k,x}(1)\right| ^2+\omega_{N}\left| \varphi^N_{k,xx}(1)\right| ^2+\omega_{0}\left| \varphi^N_{k,xx}(-1)\right| ^2}\le C.
		\end{equation}
		In Figure \ref{fig2}, we show  numerically that this constant $ C\leq 1 $.
	\end{remark}
	
	\subsection{Convergence of the discrete controls: proof of Theorem \ref{thcon1-d}}	
	
	We now prove the convergence result mentioned in Theorem \ref{thcon1-d}. All along this section we assume that the hypotheses of the theorem hold. 
	To clarify the exposition we proceed in 3 steps where we first define the sequence $(u^{0,N},u^{1,N})\in (\mathbb{P}^{Di}_{N}(\Omega))^2$ and prove the uniform boundedness of the associated sequence of controls, then characterize their weak limit and finally prove the strong convergence respectively. 
	
	{\bf Step 1: Uniform bound of the controls.}	 We first state the following result that we prove in the Appendix below. 
	\begin{lemma}\label{dataconv}
		Given $(u^{0},u^{1})\in L^{2}(\Omega)\times H^{-1}(\Omega)  $, there exists a sequence $(u^{0,N},u^{1,N})\in (\mathbb{P}^{Di}_{N}(\Omega))^2 $ such that 
		\begin{align}\label{ini}
			(u^{0,N},u^{1,N})\longrightarrow	(u^{0},u^{1})\ \mbox{in}\ L^{2}(\Omega)\times H^{-1}(\Omega).
		\end{align}
		On the other hand, for any $(\phi^{0,N},\phi^{1,N}) \in \mathbb{P}^{Di}_{N}(\Omega)\times \mathbb{P}^{Di}_{N}(\Omega)$ we have
		\begin{equation}\label{2cond}
			\left| \big<(\phi^{0,N},\phi^{1,N}),(u^{0,N},u^{1,N})\big>_{N} \right| \leq  |(u^0,u^1)|_{L^2\times H^{-1}} \| (\phi^{0,N}_x,\phi^{1,N}) \|_{N\times N}.
		\end{equation}
		Furthermore, if $(\phi^{0,N},\phi^{1,N})\to (\phi^0,\phi^1)$ in $H^1_0(\Omega)\times L^2(\Omega)$ then,
		\begin{equation}\label{2cond_b}
			\big<(\phi^{0,N},\phi^{1,N}),(u^{0,N},u^{1,N})\big>_{N} \to  \big< (\phi^0,\phi^1),(u^0,u^1) \big> ,
		\end{equation}		
		where
		\begin{align}\label{data2}
			\big<(\phi^0,\phi^1),(u^0,u^1)\big>=<u^1,\phi^0 dx>_{H^{-1},H^1_0}-\int_{-1}^1u^0\phi^1 dx,
		\end{align}  
		and $<\cdot,\cdot>_{H^{-1},H^1_0}$ is the duality product between $ H^{-1} $ and $ H^{1}_0  $.
	\end{lemma}
	
	\begin{remark}
		When $u^0$ and $u^1$ are continuous functions the polynomials $\tilde u^{0,N},\tilde u^{1,N}\in \mathbb{P}_N^{Di}$ that coincides with $u^0$ and $u^1$ at the nodes $x_i\in C^{\Omega}$ give a discretization that satisfies 
		$$
		\tilde u^{0,N}\to u^0, \quad \tilde u^{1,N}\to u^1,\quad \mbox{ in $L^2(-1,1)$,}
		$$
		as $N\to \infty$. In this case the result in Lemma \ref{dataconv} is still true when considering 
		$(u^{0,N}, u^{1,N})=(\tilde u^{0,N}, \tilde u^{1,N})$.
	\end{remark}
	
	We choose $(u^{0,N},u^{1,N})\in (\mathbb{P}^{Di}_{N}(\Omega))^2$ as in this lemma. Note that in particular this sequence is uniformly bounded in $L^2\times H^{-1}$ as $N\to \infty$ and the boundedness of the associated controls is a direct consequence of the following result:
	\begin{proposition}\label{probound1-d}
		Let $f^N,$ $g_L^N$, $g_R^N$ be the controls defined in \eqref{eq181-d}. Then, there exists a constant $ M>0 $, independent of $N$, such that 
		\begin{align}\label{boun1-d}
			\begin{split}
				&\int_{0}^{T}\left|f^{N}(t) \right| ^{2}dt\
				+\int_{0}^{T}\left|g^{N}_{R}(t) \right| ^{2}dt
				+\int_{0}^{T}\left|g^{N}_{L}(t) \right| ^{2}dt\le M\left| (u^{0},u^{1})\right|^{2}_{L^2\times H^{-1}}.
			\end{split}
		\end{align}
	\end{proposition}
	\begin{proof}[\textbf{Proof.}] As $ \hat{\phi^{N}}  $ is the solution of the  discrete adjoint system \eqref{eq1111-d} associated to the the minimizer $ (\hat{\phi}^{0,N},\hat{\phi}^{1,N}) $ of $ J^{N} $ in $ \mathbb{P}^{Di}_{N}(\Omega)\times \mathbb{P}^{Di}_{N}(\Omega) $, we have in particular
		\begin{align}
			J^{N}(\hat{\phi}^{0,N},\hat{\phi}^{1,N})\le J^{N}(0,0)=0.
		\end{align}
		Then, taking into account the conservation of the discrete energy, defined in \eqref{eq_energy}, the approximation in \eqref{ini}, estimate \eqref{2cond} and the uniform observability inequality in Lemma \ref{definv1-d} we obtain
		\begin{align*}
			&\int_{0}^{T}\eta(t)\left| \hat{\phi}^{N}_{x}(t,1)\; dt -\omega_{N}\hat{\phi}^{N}_{xx}(t,1)\right| ^{2} dt+\int_0^T\eta(t)\omega_{N}\left| \hat{\phi}^{N}_{xx}(t,-1)\right| ^{2} dt \\
			& \quad +\int_0^T\eta(t)\omega_{N}\int_{0}^{T}\left| \hat{\phi}^{N}_{xx}(t,1)\right| ^{2} dt \leq C\left|(u^{0},u^{1}) \right|_{L^2\times H^{-1}}^2,
		\end{align*}
		which is equivalent to \eqref{boun1-d}.
	\end{proof}
	
	\textbf{Step 2: Weak convergence of the control $f^N$.} Thanks to the
	bound \eqref{boun1-d} controls $ f^{N},g_{R}^{N},g^{N}_{L} $ are uniformly bounded in $ L
	^{2}(0,T) $ and therefore, there exists a subsequence, still denoted by $ f^{N},g_{R}^{N},g^{N}_{L} $, such that
	\begin{align}\label{weakc}
		f^{N}\underset{N\to\infty}{\rightharpoonup} h,\ g_{R}^{N}\underset{N\to\infty}{\rightharpoonup} h_{R},\ g_{L}^{N}\underset{N\to\infty}{\rightharpoonup} h_{L}, \  \mbox{ weakly in } \  L^{2}(0,T).
	\end{align}
	
	Let us see that $ h = f $ where $ f $ is the control with minimal $L^2-$weighted norm of system \eqref{eq1-d}, with the weigth function $\eta(t)$. In the next step we show that $h_R=h_L=0$.  This control $ f(t) $ can be characterized by the following two properties (see \cite{bib6}):
	\begin{enumerate}
		\item[(P1)] $ f $ satisfies the variational characterization of controls, i.e. 
		\begin{align}\label{varc}
			\begin{split}
				0=\int_{0}^{T}\phi_{x}(t,1)f(t) dt-\big<(\phi(0,\cdot),\phi_{t}(0,\cdot)),(u^0,u^1)\big>,
				&\; \; \forall (\phi^{0},\phi^{1}) \in H_{0}^{1}\times L^{2},\\
			\end{split}
		\end{align} 
		where $ \left\langle \cdot,\cdot\right\rangle  $ is  defined in \eqref{data2} and $ \phi $ is the solution of 
		\begin{align}\label{eq41-d}
			\begin{split}
				\begin{cases}
					\phi_{tt}-\phi_{xx}=0&\mbox{in}\ (t,x)\in (0,T)\times\Omega\\
					\phi(t,1)=\phi(t,-1)=0&\mbox{in}\ t\in (0,T)\\
					\phi(T,x)=\phi^{0},\quad \phi_{t}(T,x)=\phi^{1}&\mbox{in}\ x\in  \Omega.
				\end{cases}
			\end{split}
		\end{align}
		
		\item[(P2)] $ f(t) = \eta(t)\hat{\phi}_{x}(t,1)  $ where $ \hat{\phi} $ is a solution of adjoint continue problem \eqref{eq41-d}  with final data $ (\hat{\phi}^{0},\hat{\phi}^{1}) $.
	\end{enumerate}
	
	In what follows we see that $ h $ verifies these two properties. We start with the second one. 		
	By the boundedness of controls, the estimate \eqref{eq2441-d} and the norm equivalence in \eqref{eq10}, we deduce that $ (\hat{\phi}^{0,N},\hat{\phi}^{1,N}) $ is uniformly bounded in $ H_{0}^{1}\times L^{2} $. So, we can extract a sub-sequence,
	still denoted $ (\hat{\phi}^{0,N},\hat{\phi}^{1,N}) $ such that
	\begin{align}
		(\hat{\phi}^{0,N},\hat{\phi}^{1,N})\rightharpoonup (\hat{\phi}^{0},\hat{\phi}^{1})\ weakly \ in\ H_{0}^{1}(\Omega)\times L^{2}(\Omega).
	\end{align}
	Let $\hat{\phi}^{N}$ and $\hat{\phi}$ be the solutions of the discrete adjoint system \eqref{eq1111-d} and the continuous one \eqref{eq41-d}, associated to the final data $ (\hat{\phi}^{0,N},\hat{\phi}^{1,N}) $ and $ (\hat{\phi}^{0},\hat{\phi}^{1}) $ respectively. The following holds: 
	\begin{eqnarray}\label{str}
		&& \hat{\phi}^{N} \rightharpoonup\hat{\phi}, \mbox{ weakly-* in } \ L^{\infty}(0,T;H_{0}^{1}(\Omega))\cap W^{1,\infty}(0,T;L^{2}(\Omega)), \\ \label{str2}
		&&  \hat{\phi}^{N}_x(t,1) -\omega_{N}\hat{\phi}^{N}_{xx}(t,1) \rightharpoonup\hat{\phi}_x(t,1), \mbox{ weakly in }\ L^{2}(0,T).
	\end{eqnarray}
	
	The convergence result \eqref{str} is easily deduced from the classical theory of spectral approximation in \cite{bib5} (Section 10.5). Concerning \eqref{str2}, we write system \eqref{eq1111-d} in a weak form and take as test functions $\psi^{N}(x)=\frac{1+x}{2}\in \mathbb{P}_{N}$ and $l(t)\in C^1_0(0,T)$. 	
	Multiplying the equations \eqref{eq1111-d} by the weights $w_i$ and $ \psi^{N}(x_{i})l(t) $, adding in $i\in I$ and integrating in  time  we easily obtain the following identity, 	
	\begin{align}\label{fv}
		\begin{split}
			0&= -\int_0^T (\hat{\phi}^{N}_{tt}(t,\cdot), \frac{1+x}{2})_N l(t)dt + \int_0^T \int_{-1}^1\hat{\phi}^{N}_{xx}(t,x) \frac{1+x}{2} l(t)dt\\
			&\hspace*{0.5cm}-\int_{0}^{T}\omega_{N}\hat{\phi}^{N}_{xx}(t,1)) l(t)dt
			\\
			&=\int_{0}^{T}\int_{-1}^1\hat{\phi}^{N}_{t}(t,x) \frac{1+x}{2} l_{t}(t)dt-\frac12 \int_{0}^{T}\int_{-1}^{1}\hat{\phi}^{N}_{x}(t,x) l(t) dt dx\\&\hspace*{0.5cm}+\int_{0}^{T}(\hat{\phi}^{N}_{x}(t,1)-\omega_{N}\hat{\phi}^{N}_{xx}(t,1)) l(t)dt,\ \forall l(t) \in C_{0}^{1}(0, T).\\
		\end{split}
	\end{align}
	Note that in the second term on the right hand side we have used the quadrature formula since the integrand is a polynomial of degree $2N-2$, and this allowed us to integrate by parts in the $x$ variable.
	
	We can pass to the limit in \eqref{fv} thanks to \eqref{weakc} and \eqref{str}. 
	Then $ \hat{\phi} $ verifies
	\begin{align}\label{fv1}
		\begin{split}
			0=&\int_{0}^{T}\int_{-1}^{1}\hat{\phi}_{t}(x)\frac{1+x}{2} l_{t}(t)dx dt-\frac12 \int_{0}^{T}\int_{-1}^{1}\hat{\phi}_{x}(t,x) l(t)dx dt+\int_{0}^{T}\tilde h(t) l(t)dt,
		\end{split}
	\end{align}
	for all  $l(t)\in C_{0}^{1}(0, T). $ Here $\tilde h$ is the weak limit of $\hat{\phi}^{N}_{x}(t,1)-\omega_{N}\hat{\phi}^{N}_{xx}(t,1)$. 
	On the other hand, as $ \hat{\phi} $ is a solution of \eqref{eq41-d}, it also verifies
	\begin{align}\label{fv2}
		\begin{split}
			0=\int_{0}^{T}\int_{-1}^{1}\hat{\phi}_{t}(t,x)\frac{1+x}{2} l_{t}(t)dx dt-\frac12\int_{0}^{T}\int_{-1}^{1}\hat{\phi}_{x}(t,x) l(t)dx dt+\int_{0}^{T}\hat{\phi}_{x}(t,1) l(t)dt,
		\end{split}
	\end{align}
	for all $l(t) \in C_{0}^{1}(0, T). $ From \eqref{fv1} and \eqref{fv2} we finally deduce 
	\[\int_{0}^{T}\tilde h(t) l(t)dt=\int_{0}^{T}\hat{\phi}_{x}(t,1) l(t)dt,\ \forall \ l(t) \in C_{0}^{1}(0, T),\]
	and then $ \tilde h(t) = \hat{\phi}_{x}(t,1) $ with $ \hat{\phi} $ the solution of \eqref{eq41-d}. This finishes the proof of \eqref{str2} and in particular that $h$ satisfies property (P2) above.
	
	Now, we check that the weak limit of $f^N(t)$, $ h(t) $, also verifies the first property (P1) above.
	We need the following lemma that we prove in the appendix below:
	\begin{lemma}\label{lemcon1-d}
		Given $(\phi^{0},\phi^{1})\in H_{0}^{1}(\Omega)\times L^{2}(\Omega)  $, there exists a sequence $(\phi^{0,N},\phi^{1,N})\in \mathbb{P}^{Di}_{N}(\Omega)\times \mathbb{P}^{Di}_{N}(\Omega) $ such that 
		\begin{align}\label{eq28} 
			(\phi^{0,N},\phi^{1,N})\longrightarrow	(\phi^{0},\phi^{1}),\ in \ H_{0}^{1}(\Omega)\times L^{2}(\Omega).
		\end{align}
		Furthermore, if $ \phi^{N} $ is the solution of \eqref{eq1111-d} with final data $ (\phi^{0,N},\phi^{1,N}) $ and $ \phi $ is the solution
		of \eqref{eq41-d} with final data $ (\phi^{0},\phi^{1}) $ the following holds:
		\begin{align}\label{eqa31}
			\phi^{N} \longrightarrow &\phi, \ in \ C([0,T];H_{0}^{1}(\Omega))\cap C^{1}([0,T];L^{2}(\Omega)),
		\end{align} 
		\begin{align}\label{eqa32}  
			\phi^{N}_{x}(t,1) \longrightarrow \phi_{x}(t,1),\ in \ L^{2}(0,T),
		\end{align}
		\begin{align}\label{eqa323}  
			\sqrt{w_{N}}\phi^{N}_{xx}(t,\pm 1) \longrightarrow 0, \ in \ L^{2}(0,T),
		\end{align} 
	\end{lemma} 	
	
	Given $ (\phi^{0},\phi^{1})\in H_{0}^{1}(\Omega)\times L^{2}(\Omega) $, by Lemma \ref{lemcon1-d} there exists a sequence $(\phi^{0,N},\phi^{1,N})\in \mathbb{P}^{Di}_{N}(\Omega)\times \mathbb{P}^{\partial \Omega}_{N}(\Omega) $ such that \eqref{eq28}  holds. Furthermore, from formula \eqref{eqa31} we deduce in particular that
	\begin{align}\label{str1}
		(\phi^{N}(0,x),\phi^{N}_{t}(0,x))\longrightarrow	(\phi(0,x),\phi_{t}(0,x))\ \mbox{strongly}\;\mbox{in}\  H_{0}^{1}(\Omega)\times L^{2}(\Omega),
	\end{align}	
	where $ \phi^{N} $ is the solution of \eqref{eq1111-d} with final data $ (\phi^{0,N},\phi^{1,N}) $ and $ \phi $ is the solution
	of \eqref{eq41-d} with final data $ (\phi^{0},\phi^{1}) $. Now Passing to the limit, as $ N\longrightarrow\infty $ in formula \eqref{key11-d} and  taking into account Lemmas \ref{dataconv} and \ref{lemcon1-d} we obtain that $ h $ satisfies
	\begin{align}\label{eq55}
		\begin{split}
			\int_{0}^{T}\phi_{x}(t,1)h dt-\big<(\phi(0,\cdot),\phi_{t}(0,\cdot)),(u^0,u^1)\big>=0,
		\end{split}
	\end{align}
	$ \forall (\phi^{0},\phi^{1}) \in H_{0}^{1}(\Omega)\times L^{2}(\Omega) $,
	So $ h $ verifies property (P1) above.
	

	
	{\textbf{Step 3}: Strong convergence of the controls.} From the lower semi-continuity of the norm with respect to the weak convergence we have
	\begin{equation} \label{eq_C1}
		| f|^2_{L^2} + | h_R|^2_{L^2}+ | h_L|^2_{L^2} \leq \liminf_{N\to \infty}  | f^N|^2_{L^2} + | g_R^N|^2_{L^2}+ | g_L^N|^2_{L^2}.
	\end{equation}
	On the other hand, if
	we consider formulas \eqref{key11-d} and \eqref{varc} with $(\phi^{0,N},\phi^{1,N})=(\hat{\phi}^{0,N},\hat{\phi}^{1,N})  $ and $  (\phi^{0},\phi^{1})=(\hat{\phi}^{0},\hat{\phi}^{1}) $ respectively, we obtain:
	\begin{align}\label{stronger}
		\begin{split}
			0=\int_{0}^{T}&\eta(t)\left| \hat{\phi}^{N}_{x}(t,1)-\omega_{N}\hat{\phi}^{N}_{xx}(t,1)\right| ^{2} dt+\int_{0}^{T}\eta(t)\left| \sqrt{\omega_{N}}\hat{\phi}^{N}_{xx}(t,1)\right| ^{2} dt\\&+\int_{0}^{T}\eta(t)\left|\sqrt{\omega_{0}} \hat{\phi}^{N}_{xx}(t,-1)\right| ^{2} dt-\big<(\hat{\phi}^{N}(\cdot,0),\hat{\phi}^{N}_{t}(\cdot,0)),(u^{0,N},u^{1,N})\big>_{N},
		\end{split}
	\end{align}
	and
	\begin{align}\label{eq555}
		\begin{split}
			\int_{0}^{T}\eta(t)\left| \hat{\phi}_{x}(t,1)\right| ^{2} dt-\big<(\hat{\phi}(0,\cdot),\hat{\phi}_{t}(0,\cdot)),(u^0,u^1)\big>=0.
		\end{split}
	\end{align}
	The last term in \eqref{stronger} converges to the second term in \eqref{eq555} and therefore
	\begin{align}\label{con}
		\begin{split}
			&\int_{0}^{T}\eta(t)\left| \hat{\phi}^{N}_{x}(t,1)-\omega_{N}\hat{\phi}^{N}_{xx}(t,1)\right| ^{2} dt+\int_{0}^{T}\eta(t)\left| \sqrt{\omega_{N}}\hat{\phi}^{N}_{xx}(t,1)\right| ^{2} dt\\&\qquad+\int_{0}^{T}\eta(t)\left|\sqrt{\omega_{0}} \hat{\phi}^{N}_{xx}(t,-1)\right| ^{2} dt\longrightarrow \int_{0}^{T}\eta(t)\left| \hat{\phi}_{x}(t,1)\right| ^{2} dt \  in \ L^{2}(0,T),\ 
		\end{split}
	\end{align}
	as $ N \longrightarrow \infty$.	Now, taking into account \eqref{eq_C1}, \eqref{con} and the definition of the controls in \eqref{eq181-d} we deduce, 
	$$
	| f|^2_{L^2} + | h_R|^2_{L^2}+ | h_L|^2_{L^2} \leq |f|^2_{L^2},
	$$		
	and therefore $h_R=h_L= 0$. From \eqref{con} we also have,
	\begin{equation} \label{con_b}
		| f^N|_{L^2} \to |f|_{L^2}, \quad | h_R^N|^2_{L^2} \to 0, \quad | h_L^N|^2_{L^2} \to 0,
	\end{equation} 		
	as $N\to \infty$. 
	The strong convergence of the controls $ f^{N} \underset{N\to\infty}{\longrightarrow} f $ in $ \ L^{2}(0,T) $ is a consequence of the weak convergence stated in Step 2 above and the convergence of the norms stated in \eqref{con_b}.
	
	
	\section{The 2-d wave equation}
	
	In this section we show how to extend the numerical approximation of the control  problem \eqref{eq1} to a square domain $ \Omega=(-1,1)^2 $. The results are easily extended by separation of variables but the notation is a little cumbersome. Here we only state the main result and sketch the proof of the uniform observability result.

	Now, we assume that the control acts in the right and top sides $\Gamma=\Gamma_{1}\cup\Gamma_{2} =\{(1, x_{2})\in \mathbb{R}^{2}: -1 \le x_{2} \le 1\}\cup\{(x_{1}, 1) \in \mathbb{R}^{2}: -1 \le x_{1} \le 1\} $. The rest of the boundary is then the left and down sides $\partial \Omega \backslash \Gamma=\Gamma_{3}\cup\Gamma_{4} =\{(-1, x_2)\in \mathbb{R}^{2}: -1 \le x_2 \le 1\}\cup\{(x_1, -1) \in \mathbb{R}^{2}: -1 \le x_1 \le 1\} $.
	
	We also consider the discrete parameter $ {\bf N}=(N_{1},N_{2})\in \mathbb{N}\times \mathbb{N}$ and the Legendre Gauss-Lobatto nodes in each variable $ C=\{{\bf x}_{\bf i}=(x_{1k},x_{2m}),(0,0)\le {\bf i}=(k,m)\le (N_{1},N_{2})\} $. Note that $C$ (and some other quantities defined below) depends on $N$ but we do not make explicit this dependence in the notation to simplify. 
	We also set $ C=C^{\Omega}\cup C^{\partial \Omega} $ such that $ C^{\partial \Omega}= C^{\Gamma}\cup C^{\partial\Omega \backslash \Gamma}  $ and $ {\bf I}={\bf I}_{\Omega}\cup {\bf I}_{\partial \Omega} $ such that $ {\bf  I}_{\partial \Omega}= \cup_{k=1}^4{\bf  I}_{k} $ where $ {\bf  I}_{k}$ is the set of indexes corresponding to the collocation nodes on the boundary $ \Gamma_{k}$.
	Let $ \mathbb{P}_{\bf N}(\Omega) $ be the space of continuous functions in $ \bar{\Omega} $ which are polynomials of degree
	less than or equal to $ N_{1} $ (respectively $ N_{2}$) in the $ x_{1} $-variable (respectively in the $ x_{2} $-variable), and let $ \mathbb{P}_{\bf N}^{Di}(\Omega)  $ be the subspace of $ \mathbb{P}_{\bf N}(\Omega) $ of those functions vanishing on $\partial\Omega $.

	Now we introduce the following discrete control problem: Given $u^{0,\bf{N}}, u^{1,\bf{N}} \in \mathbb{P}^{Di}_{\bf{N}}(\Omega)$ and  $ T>0 $, find  $ f^{\bf N}\in L^2(0,T;\Gamma)$, $ g_{k}^{\bf N}\in L^2(0,T;\Gamma_{k})$, $k=1,...,4$,  such that the solution  $ u^{\bf N}\in C^{\infty} (L^{2}(0,T); \mathbb{P}_{\bf N}(\Omega)) $ of the following system:
	\begin{align}\label{d1}
		\begin{split}
			\begin{cases}
				\left( u^{\bf N}_{tt}-\Delta u^{\bf N}\right) (t,{\bf x}_{\bf i})= \sum_{k=1 }^{4}g_{k}^{\bf N}(t,{\bf x}_{\bf i})G_{k}^{\bf N}({\bf x}_{\bf i})
				\ &\mbox{in}\ (t,{\bf x}_{\bf i})\in(0,T) \times C^{\Omega}\\
				u^{\bf N}(t,{\bf x}_{\bf i})=f^{\bf N}(t,{\bf x}_{\bf i}) \ & \mbox{in}\ (t,{\bf x}_{\bf i})\in(0,T) \times C^{\Gamma} \\
				u^{\bf N}(t,{\bf x}_{\bf i})=0 \ &\mbox{in}\ (t,{\bf x}_{\bf i})\in(0,T) \times C^{\partial\Omega \backslash \Gamma} \\
				u^{\bf N}(0,{\bf x}_{\bf i})=u^{0,\bf N} ({\bf x}_{\bf i}),\quad u_{t}^{\bf N}(0,{\bf x}_{\bf i})=u^{1,\bf N} ({\bf x}_{\bf i})&\ \mbox{in} \ {\bf x}_{\bf i}\in  C^{\Omega},
			\end{cases}
		\end{split}
	\end{align}
	satisfies
	\begin{align}\label{eq14}
		\begin{split}
			u^{\bf N}(T,{\bf x}_{\bf i})=u^{\bf N}_{t}(T,{\bf x}_{\bf i})=0,\ {\bf x}_{\bf i}\in  C^{\Omega} .
		\end{split}
	\end{align}
	Here $G^{\bf N}_{k},k=1,..,4 $ are defined as in the 1-d case. For example, for the left and right boundaries, $\Gamma_3$ and $\Gamma_1$, the functions $G^{\bf N}_{3}$ and $G^{\bf N}_{1}$ depend only on the $x_1$ variable and coincide with $G^{L}(x_1)$ and $G^R(x_1)$ defined in \eqref{condc1} associated to polynomials of degree $N_1$ in $x_1$. Analogously, for the bottom and top boundaries  , $\Gamma_4$ and $\Gamma_2$, the functions $G^{\bf N}_{4}$ and $G^{\bf N}_{2}$ depend only on the $x_2$ variable and coincide again with $G^{L}(x_2)$ and $G^R(x_2)$ defined in \eqref{condc1}, this time associated to polynomials of degree $N_2$ in $x_2$.
	
	We now state the main results of existence and convergence for discrete control.

	\begin{theorem}\label{thdefc2-d}
		Given $T>4\sqrt{2}(2+N_1^{-1})(2+N_2^{-1})$ and $(u^{0,{\bf N}}, \; u^{1,{\bf N}}) \in (\mathbb{P}^{Di}_{\bf N}(\Omega))^2$, there exist controls $f^{\bf N} \in \; L^{2}(0,T;\Gamma),\; g^{{\bf N}}_{k}\in \;  L^{2}(0,T;\Gamma_{k})$, $k=1,...,4$, such that the solution $u^{\bf N}$ of 
		\eqref{d1} satisfies \eqref{eq14}.
	\end{theorem}
	
	\begin{theorem}\label{Thcon2-d}
		Given $ (u^{0}, u^{1})\in L^{2}\times H^{-1}  $, there exists a sequence $ (u^{0,\bf  N}, u^{1,\bf  N}) \in (\mathbb{P}^{Di}_{\bf N}(\Omega))^2 $such that  \[(u^{0,\bf  N}, u^{1,\bf  N})\to (u^{0}, u^{1})  \;\mbox{in}\; L^2\times H^{-1}.\]
		Furthermore, for any $ T > 4\sqrt{2}(2+N_1^{-1})(2+N_2^{-1}) $, we can choose the controls $ f^{\bf N},g_{k}^{\bf N},$ $k=1,...,4,$ such that the solution $ u^{\bf N} $ of \eqref{d1}  satisfies \eqref{eq14} and 
		\begin{align}\label{converg}
			\begin{split}
				&\hspace*{1cm}f^{\bf N}\longrightarrow f,\; in \;  L^{2}(0,T;\Gamma), \quad g_{k}^{\bf N}\longrightarrow 0,\;  in \;  L^{2}(0,T;\Gamma_{k})
			\end{split}
		\end{align}
	\end{theorem}
	
	The proofs of these two results follow closely the one-dimensional case. They are based on a suitable variational characterization of the controls and the  uniform observability inequality for the corresponding adjoint system, 
	\begin{align}\label{eq111-2d}
		\begin{cases}
			(\phi^{\bf N}_{tt}-\Delta\phi^{\bf N})(t,{\bf x}_{\bf i})=0 &\mbox{in}\ (t,{\bf x}_{\bf i})\in (0,T)\times C^{\Omega}\\
			\phi^{\bf N}(t,{\bf x}_{\bf i})=0 & \mbox{in}\ (t,{\bf x}_{\bf i})\in (0,T)\times C^{\partial \Omega}\\
			\phi^{\bf N}(T,{\bf x}_{\bf i})=\phi^{0,\bf N},\phi_{t}^{\bf N}(T,{\bf x}_{\bf i})=\phi^{1,\bf N}&\mbox{in}\ {\bf x}_{\bf i}\in C^{\Omega}.
		\end{cases}
	\end{align}		
	\begin{lemma}\label{definv2-d}
		Given $ T >4\sqrt{2}(2+N_1^{-1})(2+N_2^{-1}) $ there exists a constant $  C > 0 $, independent of $\bf N $, such that the solution of system \eqref{eq111-2d} satisfy
		\begin{align}\label{eq244}
			\begin{split}
				C\left\|(\nabla\phi^{0,\bf N},\phi^{1,\bf N}) \right\|^{2}_{\bf N\times \bf N}&\le\int_{0}^{T}\sum_{{\bf i}\in {\bf I}_{1}\cup {\bf I}_{2}}\left( \dfrac{\partial\phi^{\bf N}}{\partial\nu}-\omega^{\xi_{1}}_{\bf N}\dfrac{\partial^{2}\phi^{\bf N}}{\partial^{2}\nu}\right)^{2} (t,{\bf x}_{\bf i})\omega_{\bf i}^{\xi_{2}}dt\\&
				+\int_{0}^{T}\sum_{ {\bf i}\in {\bf I}_{\partial \Omega}}\omega^{\xi_{1}}_{\bf N}\left( \dfrac{\partial^{2}\phi^{\bf N}}{\partial^{2}\nu}\right)^{2} (t,{\bf x}_{\bf i})\omega_{\bf i}^{\xi_{2}} dt,
			\end{split}
		\end{align}
		for any final data $ (\phi^{0,\bf N},\phi^{1,\bf N})\in \mathbb{P}^{Di}_{\bf N}(\Omega)\times \mathbb{P}^{Di}_{\bf N} (\Omega)$.
		Here, $\omega^{\xi_{1}}_{\bf N}$ and $\omega^{\xi_{2}}_{\bf i}$ are defined by
		\begin{align} \label{eq_ome}
			\begin{cases}
				\omega^{\xi_{1}}_{\bf N}=\omega^{x_{1}}_{N_{1}}, \quad \omega^{\xi_{2}}_{\bf i}=\omega^{x_{2}}_{m}\; \quad \mbox{if} \ {\bf i}\in {\bf I}_{1}\cup {\bf I}_{3},\\
				\omega^{\xi_{1}}_{\bf N}=\omega^{x_{2}}_{N_{2}}, \quad \omega^{\xi_{2}}_{\bf i}=\omega^{x_{1}}_{k}\; \quad  \mbox{if} \ {\bf i}\in {\bf I}_{2}\cup {\bf I}_{4},
			\end{cases}
		\end{align}	
		and  $ \omega_{k}^{x_1}, \omega_{m}^{x_2}$ are the discrete weights associated with the one-dimensional LGL quadrature formula in each one of the variables (e.g. \cite{bib5}, Chapter 2]).	
	\end{lemma}		
	
	As in the one dimensional case, the proof of Lemma \ref{definv2-d} is based on the associated error equation, equivalent to \eqref{eq111-2d}, and given by, 
	\begin{align}\label{eq15}
		\begin{cases}
			\phi^{\bf N}_{tt}-\Delta\phi^{\bf N}=-\sum_{{\bf i}\in  I_{\partial \Omega}}\dfrac{\partial^{2} \phi^{\bf N} }{\partial^{2} \nu}(t,{\bf x}_{\bf i})\Psi^{x_1}_{i_1}\Psi^{x_2}_{i_2} &\mbox{in} \ Q=(0,T)\times\Omega  \\
			\phi^{\bf N}=0 & \mbox{on}\ (0,T)\times\ \partial\Omega\\
			\phi^{\bf N}(T,x)=\phi^{0,\bf N},\phi_{t}^{\bf N}(T,x)=\phi^{1,\bf N} & \mbox{in}\ \Omega,
		\end{cases}
	\end{align}
	where ${\bf i}=(i_1,i_2)$ and $ \Psi^{x_1}_{k}(x_1) $ (respectively $ \Psi^{x_2}_{m}(x_2) $ ) are the Lagrange polynomial which are $ 1 $ at $ x_{1k} $ (respectively at $ x_{2m} $ ) and $ 0 $ at all the other collocation points.	
	For this error equation we can apply the classical multipliers technique (see Lions \cite{bib15}).  The extra terms coming from the right hand side in \eqref{eq15} are estimated 
	following the same idea in $ 1$-d case. The analysis is sraightforward and it does not introduce new difficulties. 
	
	\section{ Numerical experiments}
	
	In this section we illustrate the results in this paper  approximating the  boundary control both for the $ 1$-d and $ 2$-d wave equation in the square.

	\textbf{Experiment 1}: we first consider the one-dimensional wave equation  with  two different types of initial position and velocity.  The first one corresponds to a smooth bump that moves to the left hand side and it is controlled from the right extreme. It is given by $ u^{0}(x)=e^{-10x^2},\; u^{1}(x)=-20xe^{-10x^2}$. 
	The second one corresponds to a Lipschitz continuous initial data  $ u^{0}(x)=min\{(1-x),(1+x)\},\; u^{1}(x)=0$ . We take final time $T = 4.4$ with time step $ dt= 10^{-2}$. Note that the time control is only slightly greater than the minimal control time for the continuous wave equation ($T=4$) and lower than the time given by the discrete control problem in Theorems \ref{thcon1-d} and \ref{thdefc1-d}. The good approximation obtained in this case provides a numerical evidence that the control time in the mentioned theorems is probably not optimal. 
	
	In table \ref{table_1} and \ref{table_2} we show the behavior of the norm of the controls when the degree of polynomials $ N $ grows. As stated in Theorem \ref{thdefc1-d}, we observe that the boundary control remains bounded while the two artificial controls included in the system ($g_L^N$ and $g_R^N$) vanish as $ N $ grows. The initial data and the boundary control are plotted in Figure \ref{fg4}.
	\begin{table}[h]
		\caption{Norm of the null controls when $ u^{0}(x)=e^{-10x^2},u^{1}(x)=-20xe^{-10x^2}$ as $N$ grows}\label{table_1}%
		\begin{tabular}{@{}llll@{}}
			\toprule
			N & $\left| f^N\right| _{L^{2}} $& $\left| g^{N}_{R}\right| _{L^{2}} $   &$\left| g^{N}_{L}\right| _{L^{2}} $\\
			\midrule
			20	&$ 5.6\times10^{-1}  $& $ 2\times 10^{-3}  $  &$ 2\times 10^{-3}  $\\ 
			50	&$ 5.6\times10^{-1}  $& $ 1\times 10^{-4}  $  &$ 1\times 10^{-4}  $\\ 
			100 & $ 5.6\times10^{-1}  $& $ 1\times 10^{-6}  $ & $ 1\times 10^{-6} $\\ 
			\botrule
		\end{tabular}
	\end{table}
	\begin{table}[h]
		\caption{ Norm of the null controls as $N$ grows for $ u^{0}(x)=min\{(1-x),(1+x)\},$  $u^{1}(x)=0$ }\label{table_2}%
		\begin{tabular}{@{}llll@{}}
			\toprule
			N & $\left| f^{N}\right| _{L^{2}} $& $\left| g^{N}_{R}\right| _{L^{2}} $   &$\left| g^{N}_{L}\right| _{L^{2}} $\\
			\midrule
			20	&$ 5.8\times10^{-1}  $& $ 3\times 10^{-4}  $  &$ 3\times 10^{-4}  $\\ 
			50 & $ 5.8\times10^{-1}  $& $ 4\times 10^{-4}  $ & $ 4\times 10^{-4}  $\\ 
			100 & $ 5.8\times10^{-1}  $ &  $ 1\times 10^{-4}  $ &$ 1\times 10^{-4}  $\\ 
			200 & $ 5.8\times10^{-1}  $ &  $ 1\times 10^{-6}  $ &$ 1\times 10^{-6}  $\\ 
			\botrule
		\end{tabular}
		
	\end{table}
	
	\begin{figure}[h]
		\centering
		\includegraphics[width=11cm]{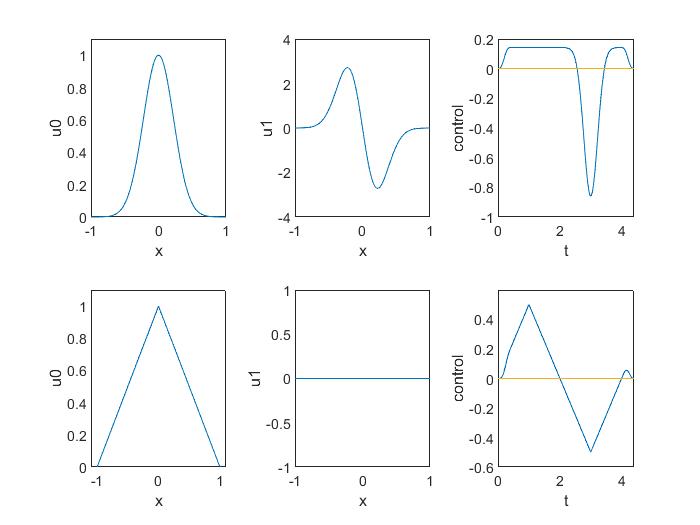}
		\caption{The behavior of discrete control $ f^N$ for $ u^{0}(x)=e^{-10x^2},u^{1}(x)=-20xe^{-10x^2}$ (first row) and for $ u^{0}(x)=min\{(1-x),(1+x)\},u^{1}(x)=0$ (second row) for $ N=100 $}
		\label{fg4}
	\end{figure}

	\textbf{Experiment 2:} Here we illustrate the rate of  convergence of the discrete control  to the limit. In particular we compare the $ L^2 $-norm of the difference between the discrete control when $ N=200 $ (that we take as continuous control) and the discrete control as $ N $ grows. The rest of the data are as in the experiment 1 when considering Lipschitz continuous initial data.
	
	In table \ref{table_3} we show that the error between the discrete control and the limit one decreases very fast recovering the high accuracy expected by the spectral collocation method, even when non-smooth data is considered. This is one of the main advantages of the collocation method. 
	\begin{table}[h]
		\caption{ Convergence of the discrete control to the limit as $N$ grows}\label{table_3}%
		\begin{tabular}{@{}llll@{}}
			\toprule
			N  & $\mbox{log}(\left| f^{N}-f^{200}\right| _{L^{2}}) $& $\mbox{log}(\left| g^{N}_{R}\right| _{L^{2}}) $   &$\mbox{log}(\left| g^{N}_{L}\right| _{L^{2}}) $\\
			\midrule
			10	&$ -1.9  $& $ -2.5  $  &$ -2.5  $\\ 
			50 & $ -2.5  $ &  $ -3.3  $ &$ -3.3 $\\
			100 &  $ -3.0 $  & $ -3.8$ & $ -3.8  $\\
			\botrule
		\end{tabular}
		
	\end{table}

	\textbf{Experiment 3}: Now we consider a two dimensional square domain $ (-1, 1)^2$. The control acts at the two sides
	$ \{1\}\times(-1, 1)\cup(-1, 1)\times\{1\}  $ in the time interval $ t \in (0,4.4),$ with step $ dt= 10^{-2}$. We consider the degree of polynomial is $ {\bf N}=(80,80) $ in the $ x_{1},x_{2}$-variable respectively. The initial position and velocity given by a bump function $ u^{0}=e^{-10x_{1}^2}e^{-10x_{2}^2} $ and $ u^{1}=(-20x_{1}e^{-10x_{1}^2})(-20x_{2}e^{-10x_{2}^2}) $. As in $ 1 $-d the time control is lower than the time given by the discrete control problem in Theorems \ref{thdefc2-d} and \ref{Thcon2-d}, and also lower than the control time for the continuous wave equation ($T=4\sqrt{2}$)(\cite{bib15}). However, the initial data is almost compactly supported in the disc $|{\bf x}|<1/2$ inside the domain and this makes this special data controllable for the chosen time. 
	In Figure \ref{Cont2d} we have drawn the behavior of the norm of control acting in the two sides of the square during the time since the other controls are of the order  $ 10^{-5} $. As in $ 1$-d in table \ref{table_4} we show the behavior of the norm for the controls when the degree of polynomials $ \bf N $ grows. As stated in Theorem \ref{Thcon2-d}, we observe that the boundary control remains bounded while the four artificial controls included in the system ($g_k^{\bf N}$, $k=1,...,4$) vanish as $ N $ grows. 
	
	\begin{figure}[h]
		\centering
		\includegraphics[width=9cm]{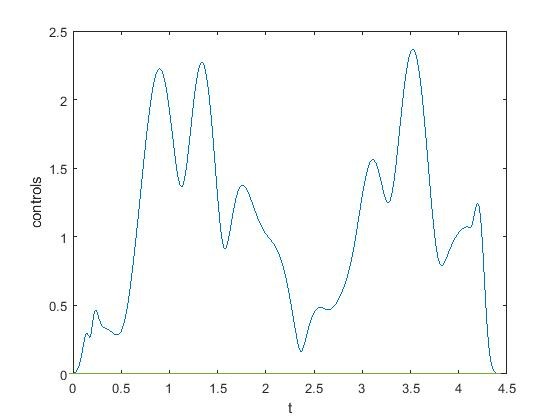}
		\caption{The behavior of control $ f^{\bf N} $ during the time for $ {\bf N}=(80,80) $}
		\label{Cont2d}
	\end{figure}
	\begin{table}[h]
		\caption{ Norm of the  controls for $ u^{0}(x)=e^{-10(x_{1}^2+x_{2}^2)}$,  $u^{1}(x)=(20^2x_{1}x_{2})$ $e^{-10(x_{1}^2+x_{2}^2)} $ as $ {\bf N} $   grows  }\label{table_4}
		\begin{tabular}{@{}llllll@{}}
			\toprule
			$ {\bf N} $ & $\left| f^{\bf N}\right| _{L^{2}} $& $\left| g^{\bf N}_{1}\right| _{L^{2}} $   &$\left| g^{\bf N}_{2}\right| _{L^{2}} $& $\left| g^{\bf N}_{3}\right| _{L^{2}} $   &$\left| g^{\bf N}_{4}\right| _{L^{2}} $\\
			\midrule
			$ (20,20) $	&$ 7.2\times10^{-1}  $& $ 8.5\times 10^{-3}  $  &$ 8.2\times 10^{-3}  $& $8.5\times 10^{-3} $   &$8\times 10^{-3} $\\ 
			$ (50,50) $ & $ 7.2\times10^{-1}  $ &  $ 7.4\times 10^{-4}  $ &$ 7.4\times 10^{-4}  $& $7.4\times 10^{-4} $   &$7.4\times 10^{-4} $\\
			$ (80,80) $&  $7.2\times10^{-1}$   & $ 6.4\times 10^{-5}$ & $ 6.4\times 10^{-5}  $& $6.4\times 10^{-5} $   &$6.4\times 10^{-5} $\\ 
			\botrule
		\end{tabular}
		
	\end{table}
	

	\section{Appendix}
	
	\hspace*{0.5cm}In this section we give a proof of Lemmas \ref{dataconv} and \ref{lemcon1-d}. Both proofs rely on a careful spectral analysis that we address first.  
	Let $ (\lambda_{k},\varphi_{k})_{k\in \mathbb{N}}  $  be the eigenvalues and the eigenfunctions associated to the Laplace equation,
	\begin{align}\label{eigp2}
		\begin{split}
			\begin{cases}
				\varphi_{xx}+\lambda\varphi=0,&\ \mbox{in} \ x\in \Omega =(-1,1) \\
				\varphi(1)=\varphi(-1)=0.
			\end{cases}
		\end{split}
	\end{align} 
	The eigenvalues are simple and can be computed explicitly ($\lambda_k=(k\pi/2)^2$, $k\in \mathbb{N}$) while the associated eigenfunctions $ \{\varphi_{k}\}_{k\in N } $ constitutes an orthogonal basis in $L^2(-1,1) $. 
	
	Associated to the collocation numerical approximation of the wave equation we introduce the following discrete eigenvalue problem:
	\begin{align}\label{eigp1}
		\begin{split}
			\begin{cases}
				\varphi^{N}_{xx}(x_{i})+\lambda^{N}\varphi^{N}(x_{i})=0,&\ \mbox{ at } \ x_{i}\in C^{\Omega} \\
				\varphi^{N}(-1)=\varphi^{N}(1)=0.
			\end{cases}
		\end{split}
	\end{align}
	It is known that this eigenvalue problem admits $ N-1 $ eigenvalues which are simple and real numbers (see Herv$ \acute{e} $ Vandeven \cite{bib18} and Trefethen \cite{bib19}). We assume they are written in increasing order, i.e. 
	$0<\lambda^{N}_{1}<\lambda^{N}_{2}<...<\lambda^{N}_{N-1}.$
	The associated eigenfunctions $ \{\varphi^{N}_{k}\}_{k=1}^{N-1} $ constitutes an orthogonal basis in $ \mathbb{P}^{Di}_{N} $ with the discrete scalar product $ (\cdot,\cdot )_{N}.$ 
	
	From now on we assume that both $\varphi_k^N$ and $\varphi_k$ are normalized in the $L^2-$norm.
	
	The following Theorem states the spectral approximation results that we need.
	\begin{theorem}\label{theigen1-d}
		Let $ m\ge 2 $, $ \alpha\in (0,2/\pi) $ and $ r(N)= \alpha N^{\frac{1}{8}}$. Then, for any $ k \in \{1,...,r(N)\}$ there exists constants $C_\alpha, \; C$, independent of $N$, such that the  following estimates hold:
		\begin{align}\label{1f}
			\left| \sqrt{\lambda^{N}_{k}}-\sqrt\lambda_{k}\right|\le C_\alpha p(\alpha)^{N^{2/3}},\ 0<p(\alpha)<1,
		\end{align}
		\begin{align}\label{2f}
			\left| \varphi^{N}_{k}-\varphi_{k}\right|_{L^{2}}\le C N^{-3/4},
		\end{align}
		\begin{align}\label{22f}
			\left| \varphi^{N}_{k}-\varphi_{k}\right|^{2}_{H^{1}_{0}}\le C\left(p(\alpha)^{N^{2/3}}+N^{-1} \right) ,
		\end{align}
		\begin{align}\label{3f} 
			\left| \sqrt{w^{N}_{0}}\varphi^{N}_{k,xx}(\pm 1)\right|\le C  N^{-1/2},
		\end{align}
		\begin{align}\label{5f} 
			\left| \varphi^{N}_{k,x}(1)-\varphi_{k,x}(1)\right|_{L^{2}}&\le C\left( p(\alpha)^{N^{2/3}}+  N^{-1}\right).
		\end{align} 
	\end{theorem}
	
	\begin{proof}[\textbf{Proof.}] 
		\textbf{Step 1: Estimate  \eqref{1f}.} This is a direct consequence of the following estimate proved in \cite{bib18}: 
		\begin{align}\label{eigcon}
			\left| \lambda^{N}_{k}-\lambda_{k}\right|\le C_\alpha p(\alpha)^{N^{2/3}}.
		\end{align}  
		
		\textbf{Step 2: Estimate  \eqref{2f}.}  We follow the idea in \cite{bib16} (Lemma 6.4-3) where a related result is obtained for the Galerkin approximation. Let us introduce the variational characterizations of both \eqref{eigp2} and \eqref{eigp1}:
		\begin{align}\label{an12}
			\begin{split}
				a(\varphi,v)=\lambda(\varphi,v), \quad \forall v\in H^{1}_{0}(\Omega),
			\end{split}
		\end{align}
		where $ a(\cdot,\cdot) $ is the bi-linear form defined by $ a(u,v)=(u_{x},v_{x}), \forall u,v\in H^{1}_{0}, $
		and $(\cdot,\cdot)$ denotes the scalar product in $L^2(-1,1)$, and 	
		\begin{align}\label{an1}
			\begin{split}
				a(\varphi^{N},v^{N})=\lambda^{N}(\varphi^{N},v^{N})_{N}, \qquad  \forall v^{N}\in \mathbb{P}_{N}^{Di}(\Omega) .
			\end{split}
		\end{align}	
		Note that the bi-linear form 	$ a(\cdot,\cdot) $ is the same both in the continuous and the discrete formulation. In fact \eqref{an1} is deduced from \eqref{eigp1} multiplying the equations by $\omega^{N}_i v^N(x_i)$ and adding in $i\in I$,
		i.e. 
		\begin{align*}
			0&=\sum_{i\in I} \varphi^N_{xx}(x_i) \omega^{N}_i v^N(x_i) + \sum_{i\in I}\lambda^N \varphi^N(x_i)\omega^{N}_i v^N(x_i)=\int_{-1}^1\varphi^N_{xx}v^N \; dx +\lambda^N  (\varphi^N,v^N)_N\\
			&=-\int_{-1}^1\varphi^N_{x}v^N_x \; dx +\lambda^N  (\varphi^N,v^N)_N=-a(\varphi^N,v^N)+\lambda^N  (\varphi^N,v^N)_N.
		\end{align*}
		The main difference with the case treated in \cite{bib16} is that  here the right hand side in the variational characterization \eqref{an1} makes appear the discrete scalar product $(\cdot,\cdot)_N$, instead of the $L^2$ one, and this introduces some technical details. 
		Let us define \begin{align}\label{Q}
			\begin{split}
				\varrho_{k,N}=\max_{i\in I_{\Omega},i\neq k} \dfrac{\lambda_{k}}{\left| \lambda^{N}_{i}-\lambda_{k}\right| },\ I_{\Omega}=\{1,...,N-1\},
			\end{split}
		\end{align}
		and the orthogonal projection $ \Pi_{N}\varphi_{k}\in \mathbb{P}^{Di}_{N}$ characterized by 
		\begin{align}\label{an}
			\begin{split}
				a(\Pi_{N}\varphi_{k}-\varphi_{k},v^{N})=0, \qquad  \forall v^{N}\in \mathbb{P}_{N}^{Di}(\Omega).
			\end{split}
		\end{align}
		We now write $ \Pi_{N}\varphi_{k} $ in
		the basis $ \varphi^{N}_{i}$. If we normalize $\varphi^{N}_{i}$ in such a way that $ |\varphi^{N}_{i}| _{L^2}=1 $, then
		\begin{align}\label{pro}
			\begin{split}
				\Pi_{N}\varphi_{k}=\sum_{i\in I_{\Omega}} \dfrac{(\Pi_{N}\varphi_{k},\varphi^{N}_{i})_{N}}{\left\|\varphi^{N}_{i} \right\|^{2} _{N}}\varphi^{N}_{i} .
			\end{split}
		\end{align}
		If we denote,
		\begin{align}\label{pro1}
			\begin{split}v^{N}_{k}=\dfrac{(\Pi_{N}\varphi_{k},\varphi^{N}_{k})_{N}}{\left\|\varphi^{N}_{k} \right\|^{2} _{N}}\varphi^{N}_{k} .
			\end{split}
		\end{align}
		Then,
		\begin{align}\label{pr-v}
			\begin{split}
				\left|\Pi_{N}\varphi_{k}-v^{N}_{k} \right|^{2}_{L^{2}}= \left| \sum_{i\in I_{\Omega},i\ne k} \dfrac{(\Pi_{N}\varphi_{k},\varphi^{N}_{i})_{N}}{\left\|\varphi^{N}_{i} \right\|^{2} _{N}}\varphi^{N}_{i}\right|_{L^{2}}^{2}
				&\le 2 \sum_{i\in I,i\ne k}   \dfrac{\left| (\Pi_{N}\varphi_{k},\varphi^{N}_{i})_{N}\right| ^{2} }{\left\|\varphi^{N}_{i} \right\|^{2} _{N}} .
			\end{split}
		\end{align}
		From \eqref{an1},\eqref{an12} and \eqref{an} we easily obtain 
		\[(\lambda^{N}_{i}-\lambda_{k})(\Pi_{N}\varphi_{k},\varphi^{N}_{i})_{N}=\lambda_{k}(\varphi_{k}-\Pi_{N}\varphi_{k},\varphi^{N}_{i})_{N}+\lambda_{k}\bigg( (\varphi_{k},\varphi^{N}_{i})-(\varphi_{k},\varphi^{N}_{i})_{N}\bigg). \]
		Therefore, using \eqref{Q} we get
		\begin{align}\label{prn}
			\begin{split}
				\left| (\Pi_{N}\varphi_{k},\varphi^{N}_{i})_{N}\right|^{2} \le2\varrho_{k,N}^{2}\left| (\varphi_{k}-\Pi_{N}\varphi_{k},\varphi^{N}_{i})_{N}\right|^{2} +2\varrho_{k,N}^{2}\left|  (\varphi_{k},\varphi^{N}_{i})-(\varphi_{k},\varphi^{N}_{i})_{N}\right|^{2},
			\end{split}
		\end{align} 
		that we now substitute in \eqref{pr-v}. Taking into account  the norm equivalence  in \eqref{eq10} we easily deduce 
		\begin{align}\label{prn3}
			\begin{split}
				\left|\Pi_{N}\varphi_{k}-v^{N}_{k}  \right|_{L^{2}}\le\sqrt{6}\varrho_{k,N}\left( \left\|\varphi_{k} - \Pi_{N}\varphi_{k} \right\|_{N}+\sum_{i\in I_{\Omega},i\ne k}\left|  (\varphi_{k},\varphi^{N}_{i})-(\varphi_{k},\varphi^{N}_{i})_{N}\right|\right).
			\end{split}
		\end{align}
		The idea now is to replace the $\| \cdot \|_N$ norm in the right hand side by the $|\cdot |_{L^2}$ norm. In fact, these two norms are equivalent for polynomials in $\mathbb{P}_N$ so that we first replace $\varphi_k$ by a polynomial. Let's define the interpolation $ I_{N}:H^{1}_{0}\longrightarrow \mathbb{P}_{N}^{Di} $ as follow 
		\begin{align}\label{In}
			\begin{split}
				\begin{cases}
					I_{N}\varphi_{k}\in \mathbb{P}_{N}^{Di},\\ 
					I_{N}\varphi_{k}(x_{i})=\varphi_{k}(x_{i}),\ \forall i\in I.
				\end{cases}
			\end{split}
		\end{align}
		Note that,  
		\begin{equation} \label{eq_cca1}
			\left\|\varphi_{k} - \Pi_{N}\varphi_{k} \right\|_{N}=\left\|I_N\varphi_{k} - \Pi_{N}\varphi_{k} \right\|_{N} \leq \sqrt{3}\left|I_N\varphi_{k} - \Pi_{N}\varphi_{k} \right|_{L^2}.
		\end{equation}
		Therefore, substituting  in \eqref{prn3} and then adding and subtracting $ \varphi_{k} $ we obtain,
		\begin{align}\label{prn5}
			\begin{split}
				\left|\Pi_{N}\varphi_{k}-v^{N}_{k}  \right|_{L^{2}}&\le\sqrt{6}\varrho_{k,N}\bigg(\sqrt{3}\left|I_{N}\varphi_{k} -\varphi_{k} \right|_{L^{2}}+\sqrt{3}\left|\varphi_{k} -\Pi_{N}\varphi_{k} \right|_{L^{2}}\\&\hspace*{0.5cm}+\sum_{i\in I_{\Omega},i\ne k}\left|  (\varphi_{k},\varphi^{N}_{i})-(\varphi_{k},\varphi^{N}_{i})_{N}\right|\bigg).
			\end{split}
		\end{align}
		Furthermore \eqref{pro1} and $\left| \varphi^{N}_{k}\right| _{L^{2}}=\left| \varphi_{k}\right| _{L^{2}}=1$ give
		\[\left|\varphi_{k}\right| _{L^{2}}-\left|\varphi_{k}- v^{N}_{k}\right| _{L^{2}}\le \left| v^{N}_{k}\right|_{L^{2}}\le\left|\varphi_{k}\right| _{L^{2}}+\left|\varphi_{k}- v^{N}_{k}\right| _{L^{2}}, \]
		and \begin{align}\label{prn52}
			\begin{split}
				\left| \dfrac{\left| (\Pi_{N}\varphi_{k},\varphi^{N}_{k})_{N}\right|}{\left\|\varphi^{N}_{k} \right\|^{2} _{N}}-1\right|\le \left|\varphi_{k}- v^{N}_{k}\right| _{L^{2}} .   \end{split}
		\end{align}
		As it is always possible to choose the eigenfunctions $ \varphi^{N}_{k} $ such that $ (\Pi_{N}\varphi_{k},\varphi^{N}_{k})_{N}\ge 0  $, we obtain
		\begin{align}\label{prn6}
			\begin{split}\left|v^{N}_{k}-\varphi^{N}_{k}\right| _{L^{2}}\le \left|\varphi_{k}- \Pi_{N}\varphi_{k}\right| _{L^{2}}+\left|\Pi_{N}\varphi_{k}- v^{N}_{k}\right| _{L^{2}} .
			\end{split}
		\end{align}
		Finally,  from \eqref{prn5} and \eqref{prn6}
		\begin{align*}
			\left|\varphi^{N}_{k}-\varphi_{k} \right| _{L^{2}}&\le 2 (1+3\sqrt{3}\varrho_{k,N})\left|\varphi_{k}-\Pi_{N}\varphi_{k} \right| _{L^{2}}+ 2 (3\sqrt{2}\varrho_{k,N})\left|\varphi_{k}-I_{N}\varphi_{k} \right| _{L^{2}}\\+&2(\sqrt{6}\varrho_{k,N})\left( \sum_{i\in I_{\Omega},i\ne k}\left| (\varphi_{k},\varphi^{N}_{i})-(\varphi_{k},\varphi^{N}_{i})_{N}\right|\right) .
		\end{align*}
		We recall that, if $v\in H_0^{m} $ for some $ m \ge 1$ and $ v^{N}\in \mathbb{P}^{Di}_{N} $, then there exist  constants $c,c_1>0  $ such that (see \cite{bib5}, [chapter(9)])
		\begin{align}\label{estimation1}
			\left| v-\Pi_{N}v\right| _{L^2}\le c N^{-m}\left| v\right| _{H^m}\ \mbox{and}\ \left| v-I_{N}v\right| _{L^2}\le c_1 N^{1/2-m}\left| v\right| _{H^m}. 
		\end{align}
		On the other hand, there exists a constant $ c_{2}>0$ such that (see \cite{bib9}[estimation (3.22)])
		\begin{align}\label{estimation}
			|(v, v^{N})_{N}-(v, v^{N})|\le c_{2}N^{-m}\left| v\right|_{H^{m}}\left|  v^{N}\right|_{L^{2}} .	
		\end{align}
		From the classical projection results for spectral methods in \eqref{estimation1} and \eqref{estimation} when $ m=2 $ and the fact that  $\left|\varphi_{k} \right|_{H^{2}}\le c \lambda_{k} \leq c N^{1/4}  $ (since by hypotheses $k\leq cN^{1/8}$) we deduce that  there exists a constant $ C>0 $ such that 	
		\begin{align}\label{prn9}
			\begin{split}
				\left|\varphi^{N}_{k}-\varphi_{k} \right| _{L^{2}}\le C N^{-3/4},\ \ k\le r(N).
			\end{split}
		\end{align}
		
		\textbf{Step 3: Estimate \eqref{22f}.} From \eqref{an12}, \eqref{an1} and the norm equivalence in \eqref{eq10} we  can write
		\begin{align*}
			a(\varphi^{N}_{k}-\varphi_{k},\varphi^{N}_{k}-\varphi_{k})\le 3 \lambda^{N}_{k}+\lambda_{k}-2\lambda_{k}(\varphi_{k},\varphi^{N}_{k}). 
		\end{align*}
		We also have
		\begin{align*}
			\left|\varphi^{N}_{k}-\varphi_{k} \right|^{2} _{L^{2}}
			&= 2(1-(\varphi_{k},\varphi^{N}_{k})), 
		\end{align*} and then
		\begin{align*}
			a(\varphi^{N}_{k}-\varphi_{k},\varphi^{N}_{k}-\varphi_{k})
			&\le 3\left( \lambda^{N}_{k}-\lambda_{K}+\lambda_{K}\left|\varphi^{N}_{k}-\varphi_{k} \right|^{2} _{L^{2}}\right). 
		\end{align*}
		From the coercivity of the bi-linear form $ a $ in \eqref{an12}, \eqref{2f} and \eqref{eigcon}, the fact that  $\lambda_{k}\le M_\alpha N^{\frac{1}{4}} $ we can deduce, there exists a constant $ C_1>0 $ such that 
		\begin{align*}
			\left|\varphi^{N}_{k}-\varphi_{k} \right|^{2} _{H_{0}^{1}}\le C_1\left(p(\alpha)^{N^{2/3}}+N^{-1} \right).
		\end{align*}
		
		\textbf{Step 4: Estimate \eqref{3f}}. It is enough to prove the estimate at $x=1$ since the other one is similar.
		First, we observe that we can write the discrete eigenvalue problem \eqref{eigp1} in the following equivalent form:
		\begin{align}\label{eigp3}
			\begin{split}
				\begin{cases}
					\varphi^{N}_{k,xx}+\lambda^{N}_{k}\varphi^{N}_{k}=\varphi^{N}_{k,xx}(-1)\Psi_{0}^N+\varphi^{N}_{k,xx}(1)\Psi_{N}^N&\ \mbox{in} \ x\in \Omega \\
					\varphi^{N}_{k}(1)=\varphi^{N}_{k}(-1)=0,
				\end{cases}
			\end{split}
		\end{align}
		where $\Psi_{0}^N\in \mathbb{P}_{N}(\Omega)$ (resp. $\Psi_{N}^N$) is the Lagrangian polynomial which is $1$ at $x=-1$ (resp. $x=1$) and $0$ at the rest of quadrature points in $C^{\Omega}=C^{\Omega,N}$. At this point we make explicit the dependence on $N$ of the set of quadrature points by writing $C^{\Omega,N}$ since we consider different sets below.  
		
		It is easy to see that the eigenfunctions associated to \eqref{eigp3} are either even or odd. We focus on the case of even eigenfunctions since the other one is similar. In this case $\varphi^{N}_{k,xx}(-1)=\varphi^{N}_{k,xx}(1)$. Multiplying \eqref{eigp3} by $ \Psi_{N-1}^{N-1} \in \mathbb{P}_{N-1}(\Omega) $  the Lagrangian polynomial which is 1 at  $ x=1 $ and 0 at all other collocation points in $C^{\Omega,N-1}$, one has
		\begin{align}\label{new way}
			\varphi^{N}_{k,xx}(1)\omega_{N-1}^{N-1}+\lambda^{N}_{k}\int_{-1}^{1}\varphi^{N}_{k}\Psi_{N-1}^{N-1}dx=\varphi^{N}_{k,xx}(1)\omega_{N}^N.
		\end{align} 
		Note that in the first term on the left hand side we have used the   quadrature formula with nodes in $C^{\Omega,N-1}$, which is exact for polynomials of degree $2(N-1)-1$. In fact, the term inside the integral is a polynomial of degree $ 2N-3 $ and by hypotheses $ \Psi^{N-1}_{N-1}  $ is $ 1 $ at $ x=1$ and $ 0 $ at the other quadrature nodes. On the right hand side we have used the quadrature formula with nodes in $ C^{\Omega,N} $  which is also exact since the integrated is a polynomial of degree $ 2N-1 $.  Therefore
		\begin{align} \label{new way1}
			\begin{split}
				\left| \sqrt{\omega^{N}_{N}}\varphi^{N}_{k,xx}(1)\right| &=\left| \dfrac{\sqrt{\omega^{N}_{N}}\lambda^{N}_{k}\int_{-1}^{1}\varphi^{N}_{k}\Psi^{N-1}_{N-1}dx}{(w^{N-1}_{N-1}-w^{N}_{N})}\right|\\
				&=\left| \dfrac{\sqrt{\omega^{N}_{N}}\lambda^{N}_{k}\int_{-1}^{1}(\varphi^{N}_{k}-P_{N-2}\varphi^{N}_{k})\Psi^{N-1}_{N-1}dx}{(\omega^{N-1}_{N-1}-\omega^{N}_{N})}\right|,
			\end{split}
		\end{align}
		where $ P_{N-2}\varphi^{N}_{k} $ is  the orthogonal projection of $ \varphi^{N}_{k} $ in $ \mathbb{P}^{Di}_{N-2}(\Omega) $ with respect to the $L^2-$scalar product. The last equality comes from the fact that $ \int_{-1}^{1} P_{N-2}\varphi^{N}_{k}\Psi^{N-1}_{N-1}dx=0$ by the  quadrature formula with nodes in $C^{\Omega,N-1}$. 
		
		Now using Cauchy schwarz inequality in \eqref{new way1} and taking into account the norm equivalent in \eqref{eq10}, as  $ \Psi^{N-1}_{N-1}\in \mathbb{P}_{N-1}(\Omega) $,$ \left|\Psi^{N-1}_{N-1} \right| _{L^{2}}\le \left\|\Psi^{N-1}_{N-1} \right\| _{N-1} =\sqrt{\omega^{N-1}_{N-1}}$ and   $\omega^{N-1}_{N-1}=\dfrac{2}{(N-1)N},\omega^{N}_{N}=\dfrac{2}{N(N+1)},  $ from \eqref{estimation1} and the fact that $ \lambda^{N}_{k}\le  M_\alpha N^{\frac{1}{4}} $  we obtain, there exists a constant $ c_3>0 $ such that   
		\begin{align}\label{prop2}
			\begin{split}
				\left| \sqrt{\omega^{N}_{N}}{\varphi^{N}_{k,xx}}(1)\right|
				&\le\dfrac{\sqrt{\omega^{N}_{N}}\sqrt{\omega^{N-1}_{N-1}}\lambda^{N}_{k}  \left| \varphi^{N}_{k}-P_{N-2}\varphi^{N}_{k}\right| _{L^{2}}  }{\left| (\omega^{N-1}_{N-1}-\omega^{N}_{N})\right| }\\
				&\le c_3N^{5/4}(N-2)^{-m}  \left| \varphi^{N}_{k}\right| _{H^{m}}.  
			\end{split}
		\end{align}
		To estimate this last term we use the following result,	
		\begin{lemma}\label{lemeigen1-d}
			Assume that $ m=2 $, then there exists a constant $ M_2>0 $ such that
			\begin{align}\label{H2}
				\left| \varphi^{N}_{k}\right|^{2} _{H^{2}} \le M_2 |\lambda^{N}_{k}|^2(1+w^{N}_{N}\left|{\varphi^{N}_{k,xx}}(1)\right|^{2}).
			\end{align}
			
		\end{lemma}
		From Lemma \ref{lemeigen1-d}, the fact that $ \lambda^{N}_{k}\le  M_\alpha N^{\frac{1}{4}} $ and \eqref{prop2} we easily deduce estimate \eqref{3f}.		
		
		We now prove Lemma \ref{lemeigen1-d}.
		\begin{proof}[\textbf{Proof.}] 
			Multiply \eqref{eigp3} by $ \varphi^{N}_{k} $ and integrating by parts one has
			\begin{equation*}
				\int_{-1}^{1}\left| \varphi^{N}_{k,x}\right| ^{2} dx=\lambda^{N}_{k}\int_{-1}^{1} \left| \varphi^{N}_{k}\right| ^{2} dx-2\varphi^{N}_{k,xx}(1)\int_{-1}^{1}\Psi^{N}_{N} \varphi^{N}_{k} dx.
			\end{equation*} 
			The last equality comes from the fact that $ \int_{-1}^{1}\Psi^{N}_{0} \varphi^{N}_{k} dx=\int_{-1}^{1}\Psi^{N}_{N} \varphi^{N}_{k} dx $. Now multiplying and dividing the second term on the right hand side by $ \sqrt{\omega^{N}_N} $, first using young's inequality then Cauchy schwarz inequality we obtain
			\begin{align}\label{H2x}
				\left| \varphi^{N}_{k,x}\right|_{L^2} ^{2} \le\left( \lambda^{N}_{k}+\frac{\left|\Psi^{N}_{N} \right|^{2} _{L^{2}}}{\omega^{N}_{N}}\right) \left| \varphi^{N}_{k}\right|^{2}_{L^{2}} +\omega^{N}_{N}\left|\varphi^{N}_{k,xx}(1) \right| ^{2}.
			\end{align}
			On the other hand  multiply \eqref{eigp3} by $ \varphi^{N}_{k,xx} $ and integrating by parts one obtains
			\begin{equation}\label{H2xx}
				\left| \varphi^{N}_{k,xx}\right|_{L^2} ^{2} =\lambda^{N}_{k} \left|\varphi^{N}_{k,x}\right|_{L^2} ^{2} +2\omega^{N}_{N}\left| \varphi^{N}_{k,xx}(1)\right| ^{2}.
			\end{equation}
			Here on the right hand side we have used the quadrature formula with nodes in $ C^N $  since the integrated is a polynomial of degree $ 2N-2 $ and by hypotheses $\Psi_{0}^N$ (resp. $\Psi_{N}^N$) is $1$ at $x=-1$ (resp. $x=1$) and $0$ at the rest of quadrature points in $C^N$ and the fact that $ \varphi^{N}_{k} $ is even. Finally from \eqref{H2x},\eqref{H2xx} and the normalization of the eigenfunctions $\left| \varphi^N_{k}\right| _{L^{2}}=1$ we easily obtain \eqref{H2}.  
		\end{proof}
		\textbf{Step 5: Estimate \eqref{5f}}. 
		Multiplying the equation of $ \varphi^{N}_{k}  $ in \eqref{eigp1} by $ \dfrac{1+x_{i}}{2}\omega_{i}^N $ and adding in $ i\in I $ one obtains,
		\begin{align}\label{def1}
			\begin{split}
				0&=\sum_{i\in I} \varphi^{N}_{k,xx}(x_{i})\dfrac{1+x_{i}}{2} \omega^{N}_{i}+\lambda^{N}_{k}\sum_{i\in I}\varphi_{k}^N(x_{i})\dfrac{1+x_{i}}{2} \omega^{N}_{i}-\varphi^{N}_{k,xx}(1)\omega^{N}_{N}\\
				&=\varphi^{N}_{k,x}(1)+\lambda^{N}_{k}\int_{-1}^{1}\varphi^{N}_{k}\dfrac{1+x}{2}dx -\varphi^{N}_{k,xx}(1)\omega^{N}_{N}.
			\end{split}
		\end{align}
		Note that in the first and second terms on the right hand side we have used the quadrature formula. In particular, this allowed us to integrate by parts the first term.
		
		Now multiplying the equation in \eqref{eigp2} by $ \dfrac{1+x}{2} $ and integrating by parts one has,
		\begin{align}\label{def2}
			\begin{split}
				\varphi_{k,x}(1)+\lambda_{K}\int_{-1}^{1}\varphi_{k}\dfrac{1+x}{2}dx=0.
			\end{split}
		\end{align}
		Combining  \eqref{def1}, \eqref{def2} we obtain
		$$
		\varphi^{N}_{k,x}(1)-\varphi_{k,x}(1)=\lambda_{k}\int_{-1}^{1}\varphi_{k}\dfrac{1+x}{2}dx - \lambda^{N}_{k}\int_{-1}^{1}\varphi^{N}_{k}\dfrac{1+x}{2}dx + \varphi^{N}_{k,xx}(1)w_{N}^N.
		$$
		Here, the right hand side is easily estimated using \eqref{2f},\eqref{eigcon} and  the normalization of the eigenfunctions $\left| \varphi^N_{k}\right| _{L^{2}}=1$. This gives \eqref{5f} and concludes the proof of the theorem. 
	\end{proof}
	
	We now move to give a proof of Lemma \ref{lemcon1-d} and Lemma \ref{dataconv}. 
	
	\textbf{Proof of Lemma \ref{lemcon1-d}}.
	Define $\mu_k=sign(k) \sqrt{\lambda_{|k|}}$  and $\Phi_k=(\varphi_{|k|} / (i\mu_k ),\varphi_{|k|})/\sqrt{2}$ for $k\in \mathbb{Z}^*=\mathbb{Z}\backslash \{ 0 \}$. Note that $\{\Phi_k\}_{k\in \mathbb{Z}}$ is an orthonormal basis in 	$H_{0}^{1}\times L^{2}$.
	Thus, given $ (\phi^{0},\phi^{1})\in H_{0}^{1}\times L^{2} $ we can write
	\begin{equation}
		(\phi^{0},\phi^{1})=\sum_{k\in \mathbb{Z}^*} a_k \Phi_k, \quad | (\phi^{0},\phi^{1})|_{H^1_0\times L^2}^2 = \sum_{k\in \mathbb{Z}^*}|a_k|^2<\infty,
	\end{equation}
	for some Fourier coefficients $a_k\in \mathbb{C}$. 
	Analogously, we define $\mu^{N}_{k}=sign(k) \sqrt{\lambda^{N}_{|k|}}$  and $\Phi_k^N=(\varphi_{|k|}^N / (i\mu^{N}_{k} ),\varphi_{|k|}^N)/\sqrt{2}$ for $|k|\leq N, \; k\neq 0$. Again $\{\Phi_k^N\}_{|k|\leq N}$ is an orthonormal basis of $H_0^1 \times \mathbb{P}_N$ where the scalar product in $\mathbb{P}_N $ is the discrete inner product $ (\cdot,\cdot)_N $. Let us consider 
	\begin{align}\label{pfr1}
		(\phi^{0,N},\phi^{1,N})= \sum_{|k|\leq r(N)} a_k \Phi_k^N. 
	\end{align}
	From the convergence results in Theorem \ref{theigen1-d} we have 
	\begin{eqnarray} \nonumber
		&& |(\phi^{0},\phi^{1})-(\phi^{0,N},\phi^{1,N})|_{H^1_0\times L^2}^2 \leq \sup_{|k|\leq r(N)}\left(| \Phi_k-\Phi_k^N |_{H_0^1\times L^2}^2 \right)\sum_{|k|\leq r(N)} |a_k|^2  \\ \label{eq_app1}
		&& \quad +\sum_{|k|> r(N)} |a_k|^2 \to 0, \quad \mbox{as $N\to \infty$} .
	\end{eqnarray}
	This concludes the proof of \eqref{eq28}. 
	Moreover, the solution of the continuous wave equation \eqref{eq41-d} is given by 
	\begin{equation} \label{eq_n1}
		(\phi(t,x),\phi_t(t,x))= \sum_{k\in \mathbb{Z}^*} a_k e^{i\mu_k t} \Phi_k,
	\end{equation}
	while the one associated to \eqref{eq1111-d} with initial data $(\phi^{0,N},\phi^{1,N}) $ is given by
	\begin{equation} \label{eq_n2}
		(\phi^N(t,x),\phi_t^N(t,x))= \sum_{|k|\leq r(N)} a_k e^{i\mu^{N}_{k} t} \Phi_k^N.
	\end{equation}
	Again, the uniform convergence of the low frequencies stated  in Theorem \ref{theigen1-d} allows us to obtain \eqref{eqa31}-\eqref{eqa323}. 
	
	\textbf{Proof of Lemma \ref{dataconv}}.  We follow the idea in the proof of Lemma \ref{lemcon1-d}. 
	Let us define  $\hat \Phi_k=(i\mu_k )\Phi_k$ for $k\in \mathbb{Z}^*=\mathbb{Z}\backslash \{ 0 \}$ where $\Phi_k$ where introduced at the beginning of Lemma \ref{lemcon1-d}. Note that $\{\hat \Phi_k\}_{k\in \mathbb{Z}}$ is now an orthonormal basis in 	$ L^{2}\times H^{-1}$.
	Thus, given $ (u^{0},u^{1})\in L^{2}\times H^{-1} $ we can write
	\begin{equation}\label{pfr22}
		(u^{0},u^{1})=\sum_{k\in \mathbb{Z}^*} \hat b_k \hat\Phi_k, \quad | (u^{0},u^{1})|_{L^{2}\times H^{-1}}^2 = \sum_{k\in \mathbb{Z}^*}|\hat b_k|^2<\infty,
	\end{equation}
	for some Fourier coefficients $\hat b_k\in \mathbb{C}$. Analogously, we define $\hat \Phi^{N}_k=i\mu^{N}_{k}\Phi^{N}_k$ for $|k|\leq N, \; k\neq 0$. Let us consider 
	\begin{align}\label{pfr2}
		(u^{0,N},u^{1,N})= \sum_{|k|\leq r(N)} \hat b_k \hat \Phi^{N}_k. 
	\end{align}
	Note that if $(u^{0}, u^{1})$ are continuous functions, the sequence that we choose $(u^{0,N}, u^{1,N})$ are the polynomial which coincides with the value of $(u^{0}, u^{1})$ at the collocation points. 
	
	Arguing as in \eqref{eq_app1} the convergence result in \eqref{ini} can be reduced to prove 
	$$
	\sup_{|k|\leq r(N)}\left(|  \hat \Phi_k- \hat \Phi^{N}_k |_{L^2\times H^{-1}}^2 \right) \to 0, \quad \mbox{ as $N\to  \infty$}.
	$$
	Note that,
	$$
	|  \hat \Phi_k- \hat \Phi^{N}_k |_{L^2\times H^{-1}}^2=|\varphi_{|k|}-\varphi_{|k|}^N|_{L^2}^2 + |i\mu_k\varphi_{|k|}-i\mu_k^N\varphi_{|k|}^N|_{H^{-1}}^2 .
	$$
	The first term here can be estimated uniformly for $|k|\leq r(N)$ by Theorem \ref{theigen1-d} and converges to zero as $N\to \infty$. Concerning the second term we use the fact that the eigenfunctions $\varphi_{k}$ (resp. $\varphi_{|k|}^N$) satisfy \eqref{eigp1} (resp. \eqref{eigp3}) together with the isometry of the Laplacian between $H^1_0$ and $H^{-1}$. Therefore,
	\begin{eqnarray}
		&& |i\mu_k\varphi_{|k|}-i\mu_k^N\varphi_{|k|}^N|_{H^{-1}}^2 =\left| \frac{\varphi_{|k|,xx}}{\mu_k}-\frac{\varphi_{|k|,xx}^N-\varphi_{|k|,xx}^N(-1)\Psi_0^N-\varphi_{|k|,xx}^N(1)\Psi_N^N}{\mu_k^N}  \right|_{H^{-1}}^2 \\
		&& \leq \left| \frac{\varphi_{|k|}}{\mu_k}-\frac{\varphi_{|k|}^N}{\mu_k^N}\right|_{H^{1}_0}^2 + \frac{|\varphi_{|k|,xx}^N(-1)|^2}{\lambda_k^N}|\Psi_0^N|_{H^{-1}}^2 + \frac{|\varphi_{|k|,xx}^N(1)|^2}{\lambda_k^N}|\Psi_N^N|_{H^{-1}}^2,
	\end{eqnarray}
	that converges uniformly to zero for $|k|\leq r(N)$ as a consequence of Theorem \ref{theigen1-d} and the uniform bound of $|\Psi_0^N|_{L^2}$ and $|\Psi_0^N|_{L^2}$. 
	
	We now prove \eqref{2cond}, Observe that, 
	$\{\Phi_k^N\}_{|k|\leq N} $ is orthonormal in $\mathbb{P}^N \times \mathbb{P}^N$ with the scalar product,
	$$
	\left( (v^{0,N},v^{1,N}), (w^{0,N},w^{1,N}) \right)_{N}^*= (v^{0,N}_x,w^{0,N}_x)_N +(v^{1,N},v^{1,N})_N, 
	$$
	whose associated norm is equivalent to the usual norm in $ H^1_0\times L^2$. 
	
	Therefore, if write any $(\phi^{0,N},\phi^{1,N})$ as $\sum_{|k|\leq N} a_k^N \Phi_k^N $ and by the orthogonality of the eigenfunctions $\varphi_k^N$ with respect to the discrete scalar product $(\cdot,\cdot)_N$ and the duality product  \eqref{dualityN} we have
	\begin{align*}
		& \left| \big<(\varphi^{0,N},\varphi^{1,N}),(u^{0,N},u^{1,N})\big>_{N} \right| = 
		\left| \big< \sum_{|k|\leq r(N)} \hat b_k \hat \Phi_k^N,  \sum_{|k|\leq N} a_k^N \Phi_k^N   \big>_{N}  \right| \\
		& \le \left| \sum_{|k|\leq r(N)}  \hat b_k a_k^N  \right| \leq 
		\left( \sum_{|k|\leq r(N)} |\hat b_k|^2 \right)^{1/2} \left( \sum_{|k|\leq r(N)} |a_k^N|^2 \right)^{1/2} \\
		& \leq |(u^0,u^1)|_{L^2\times H^{-1}} \| (\varphi^{0,N}_x,\varphi^{1,N}) \|_{N\times N}. 
	\end{align*}
	
	Finally, we prove \eqref{2cond_b}. We assume now that $(\phi^{0,N},\phi^{1,N})\to (\phi^{0},\phi^{1})$ in $H^1_0\times L^2$ that we write as $(\phi^{0},\phi^{1})=\sum_{k\in \mathbb{Z}^*} a_k \Phi_k$. 
	We have,
	\begin{align*}
		& \big<(\varphi^{0,N},\varphi^{1,N}),(u^{0,N},u^{1,N})\big>_{N}= \big< \sum_{|k|\leq r(N)} \hat b_k \hat \Phi_k^N,  \sum_{|k|\leq N} a_k^N \Phi_k^N   \big>_{N},\\
		&\big<(\varphi^{0},\varphi^{1}),(u^{0},u^{1})\big>= \big< \sum_{k\in \mathbb{Z}^*} \hat b_k \hat \Phi_k,  \sum_{k\in \mathbb{Z}^*} a_k \Phi_k   \big>.
	\end{align*}
	The convergence results in Theorem \ref{theigen1-d} allow to prove the estimate,
	$$
	|a_k^N - a_k| \leq CN^{-1/4}, \quad |k|\leq r(N),
	$$
	and, using the strong convergence $(\phi^{0,N},\phi^{1,N})\to (\phi^{0},\phi^{1})$,
	$$
	\big< \sum_{|k|\leq r(N)} \hat b_k \hat \Phi_k^N,  \sum_{|k|\leq N} a_k^N \Phi_k^N   \big>_{N} \to  \big< \sum_{k\in \mathbb{Z}^*} \hat b_k \hat \Phi_k,  \sum_{k\in \mathbb{Z}^*} a_k \Phi_k   \big>. 
	$$ 	
	This concludes the proof of \eqref{2cond_b}.
	
	\section*{Acknowledgements}
	The authors were supported by grant PID2021-124195NB-C31 from the Spanish government (MICINN). The first author also thanks the support of the Algerian government for the scholarship offered to finance the PhD at the Polytechnic University of Madrid.

	\bibliography{sn-bibliography}

\end{document}